\documentclass{amsart}
\usepackage{amssymb}
\usepackage{amsmath}
\usepackage{color}
\usepackage{graphicx}
\usepackage{hyperref}
\usepackage{stmaryrd}
\usepackage{dsfont}

\newtheorem{theo}{\indent Theorem}[section]
\newtheorem{prop}[theo]{\indent Proposition}

\newtheorem{rem}[theo]{\indent Remark}
\newtheorem{lem}[theo]{\indent Lemma}

\newtheorem{cor}[theo]{\indent Corollary}

\newtheorem{ass}[theo]{\indent Assumption}

\def\D{{\mathbb D}}

\def \E{\mathbb E}
\def \P{\mathbb P}

\newcommand{\R}{\mathbb {R}}
\newcommand{\N}{\mathbb {N}}

\def\indiq{{\bf 1}}

\newcommand{\cco}{\llbracket}
\newcommand{\ccf}{\rrbracket}
\newcommand{\po}{\left(}
\newcommand{\pf}{\right)}
\newcommand{\co}{\left[}
\newcommand{\cf}{\right]} 
\newcommand{\jump}{\bar\lambda}

\newlength{\breite}
\title{Metastability for systems of interacting neurons}
\date{}
\author{Eva L\"ocherbach \and Pierre Monmarch\'e}

\address{E. L\"ocherbach: SAMM, Statistique, Analyse et Mod\'elisation Multidisciplinaire, Universit\'e Paris 1 Panth\'eon-Sorbonne, EA 4543 et FR FP2M 2036 CNRS, France.}
\email{eva.locherbach@univ-paris1.fr}
\address{Pierre Monmarch\'e: LJLL-UMR 7598, Sorbonne Universit\'e, France.}
\email{pierre.monmarche@sorbonne-universite.fr}

\begin{document}
\maketitle
\def\abstractname{Abstract}
\begin{abstract}
We study  a stochastic system of interacting neurons and its metastable properties. The system consists of $N$ neurons, each spiking randomly with rate depending on its membrane potential.
At its spiking time, the neuron potential is reset to $0$ and all other neurons
receive an additional amount $h/N$ of potential. In between successive spike times, each neuron looses potential at exponential speed. We study this system in the supercritical regime, that is, for sufficiently high values of the synaptic weight $h.$ Under very mild conditions on the behavior of the spiking rate function in the vicinity of $0$, is has been shown in Duarte and Ost  \cite{do} that the only invariant distribution of the finite system is the trivial measure $ \delta_{\bf 0}$ corresponding to extinction of the process. We strengthen these conditions to  prove that for large synaptic weights $h,$ the extinction time arrives at exponentially late times in $ N$, and discuss the stability of the equilibrium $\delta_{\bf 0}$ for the non-linear mean-field limit process depending on the parameters of the dynamics.  We then specify our study to the case of saturating spiking rates and show that, under suitable conditions on the parameters of the model, 1) the non-linear mean-field limit admits a unique and globally attracting non trivial equilibrium and 2) the rescaled exit  times for the mean spiking rate of a finite system from a neighbourhood of the non-linear equilibrium rate converge in law to an exponential distribution, as the system size diverges. In other words, the system exhibits a metastable behavior. 
\end{abstract}

\\

{\it MSC 2000}  : 60 G 55, 60 J 25, 60 K 35

\section{Introduction}
In this paper we study the metastable behavior of a  microscopic stochastic model describing a large network of $N$ spiking neurons. Each neuron emits action potentials (spikes) at a rate $\lambda (u) $ depending on its membrane potential value $u.$ At the spiking time, the neuron's potential is reset to a resting value, which we choose equal to zero in this article. At the same time all its postsynaptic neurons receive an additional amount of potential $ h/ N ,$ where $ h > 0 $ is the synaptic weight and $ N$ the size of the system. Finally, in between successive jumps, each neuron's potential undergoes some leak effect and looses potential at exponential rate $ \alpha > 0 .$ Introduced in a discrete-time framework by Galves and L\"ocherbach in \cite{ae}, this model and its mean-field limits have been studied in De Masi, Galves, L\"ocherbach and Presutti \cite{aaee}, Fournier and L\"ocherbach \cite{fournier}, Robert and Touboul \cite{robert-touboul}, Cormier, Tanr\'e and Veltz \cite{cormier} and
Duarte and Ost \cite{do}.  \cite{fournier} and \cite{robert-touboul} propose also a discussion of the longtime behavior of the associated mean-field limits, proving in particular that for sufficiently high values of the synaptic interaction strength $h, $  the trivial measure is not attracting for the limit process. However, Duarte and Ost in \cite{do} show that, under very mild conditions on the spiking rate function $\lambda $,  the system goes extinct almost surely in finite time, that is, there exists a finite last spiking time after which the system does not present any spiking activity any more. For large systems, the system is expected to mimic the behavior of the limit system over long time intervals and to stay close to a temporary equilibrium state, the {\it metastable state}, before finally being kicked out of the metastable state and going rapidly to extinction. The present article formalizes this idea in mathematical terms. One of our main results is that for spiking rate functions that saturate and grow linearly before saturation, the mean spiking  rate of the system stays in the vicinity of the limit equilibrium for a time that, rescaled by its expected value, converges to an exponential distribution as $N$ goes to infinity. By the memoryless property of the exponential distribution, it means the exit time is unpredictable. This is what is commonly called {\it metastable behavior}. 

Metastability is a widely studied subject nowadays, and the existence of related phenomena is conjectured to play an important role in nature, in particular in brain functioning and the ability of systems of neurons to process information (\cite{deco}). It is also a major issue for stochastic algorithms (see \cite{lelievre} and references within). A metastable system stays in the neighborhood of a seemingly stable state, the metastable state, during a very long random time period, before leaving this region of the state space at some random exit time which is  exponentially large. Both in mathematical physics and in probability, a lot of papers are devoted to the study of such phenomena, in different processes, and following very different approaches. To cite just a few of them,  the potential theoretical approach focuses on the precise analysis of the exit probabilities and the exit times of the associated sets (see the recent monograph \cite{bovierbook} and the references cited therein, see also the recent \cite{bianchi} among others). This approach is particularly efficient for the study of reversible diffusions in an energy landscape, in the small noise regime (see  \cite{nils} for an overview). Other mathematical approaches aim at identifying metastable behavior by the means of martingale problems (see \cite{landim}) or using renormalization techniques (\cite{scoppola}). Finally, another important research direction makes us of large deviation techniques, starting probably with the work of Freidlin and Wentzell on random perturbations of dynamical systems, \cite{FW}, leading to a pathwise approach where probabilities of trajectories are evaluated, identifying the most likely paths and controlling the associated probabilities. This approach has inspired, among others, one of the first papers devoted to the study of the contact process, \cite{marzio}, and  the monograph \cite{Eulalia} is devoted to this topic. 

Our paper belongs to the class of papers where the path-wise approach is adopted. In particular, we rely heavily on coupling techniques and large deviation estimates, adapting the results of Brassesco, Olivieri and Vares in \cite{brassesco} to our frame. To quote \cite{sohier}, we show in particular the three main ingredients that the authors identify therein : The state space of the process can be divided into three subdomains $\mathcal K \subset \mathcal D $ and $ (\mathcal D)^c  $ where $\mathcal K$ is a {\it trap} (in our case, a vicinity of the limit equilibrium spiking rate or the set where the total spiking rate is lower bounded by a fixed threshold) such that we have 

$-$ {\it fast recurrence}, meaning that the process enters after some controlled time either in $\mathcal K$ or in ${\mathcal D}^c .$

$-$ {\it slow escape}, meaning that starting from configurations in $\mathcal K,$ the time the process takes to hit ${\mathcal D}^c  $ is much larger than the recurrence time.

$-$ {\it fast thermalization}, meaning that a process started in $ \mathcal K$ looses memory in a time  much shorter than the escape time.

For  systems of interacting and spiking neurons, close to our model, metastability has been first addressed by  Brochini and Abadi \cite{BA}. They study a  simplified and time discrete version of our model and do not prove the asymptotical exponentiality of the rescaled exit times. Two recent papers \cite{morgan} and \cite{morgan2} prove the asymptotical exponentiality of the rescaled extinction times within a model of interacting neurons which is reminiscent of the contact process in dimension one and thus only loosely related to our model. The main point of  these  papers is to make use of  the additivity of the process which implies in particular the existence of an associated dual process -- obviously, such techniques  are not applicable in our context. Finally, for the very widely studied contact process and its metastable properties, let us also cite \cite{Schonmann} or \cite{bruno} for one of the more recent contributions.

Let us now describe our results more in detail, together with the organisation of the paper. The model is introduced and the main results are stated in Section~\ref{Section-modeldef}. Section~\ref{sec:boundL} is devoted to the proof of   a lower bound on the extinction time. In a first step we show in Proposition \ref{prop:Z} that under minimal assumptions on the behavior of the spiking rate close to $ 0 $  it is possible to introduce a simple auxiliary Markov process $ Z^N $ for which the large $N-$dynamics, in particular Large Deviation results, are easily obtained, and which provides a lower bound on the total spiking rate of the system. For sufficiently large values of $\lambda' (0) h,$  the limit process associated to the large $N-$asymptotics of $ Z^N$ possesses a unique attracting equilibrium which is strictly positive. As a consequence, using Large Deviation techniques, the extinction time of the original process  is exponentially large in $ N $ (see Proposition \ref{prop:ldp1} and Theorem~\ref{thm:LDPLeta}). Our next section, Section \ref{sec:4}, is devoted to the study of the longtime behavior of the true, nonlinear in the sense of McKean, limit process associated to the original particle system. This process has already been studied in a slightly different form in \cite{fournier}, \cite{robert-touboul} and in \cite{cormier}. Theorem \ref{theo:41} and Proposition~\ref{prop:delta0instable1} state that, if $\lambda'(0)h>\alpha$, then the trivial invariant measure corresponding to extinction is unstable and the limit process always admits at least a second non-trivial and absolutely continuous invariant measure. The proofs of these results are given in this section. The main Theorem \ref{theo:gattractive} proven in this section shows then that for piecewise linear rate functions that saturate and for sufficiently large values of $\lambda'(0)h,  $ this second invariant measure is unique and globally attracting. We then continue the study of the metastable behavior of the finite-size system. In a first step, Section~\ref{sec:generalresult} collects general conditions that ensure that the rescaled exit times of a  Markov process are close in law to the exponential law, extending the results of Brassesco, Olivieri and Vares in \cite{brassesco}  for low-noise diffusion processes to our frame. In particular, Theorem \ref{thm:gene-exit} gives error bounds for the difference of the distribution function of the rescaled exit times and the one of the exponential law.  
Section \ref{sec:metastableneurons} then collects all preceding results and applies them to the system of interacting neurons we are interested in. In particular, our main result, Theorem \ref{theo:exitTimes}, is proven here. It shows that, provided $\lambda'(0)h$ is large enough, the exit times associated to some relevant domains, rescaled by their expectation, converge in law to an exponential law of parameter one. These relevant domains are on the one hand the set where where the total spiking rate is lower bounded by a fixed strictly positive level, and on the other hand the domain of the state space where the total spiking rate is within a neighbourhood of the non-linear equilibrium rate.  As a consequence, we have proven that the process exhibits a metastable behavior. 

In the case where $\lambda(u)= (ku)\wedge \lambda_*$ for some $k,\lambda_*>0$, the results depending on the parameters $a=\alpha/(kh)$ and $b=\lambda_*/(kh)$ are gathered in  Figure~\ref{fig-parametre}.

  \begin{figure}
\begin{center}
\includegraphics[scale=0.18]{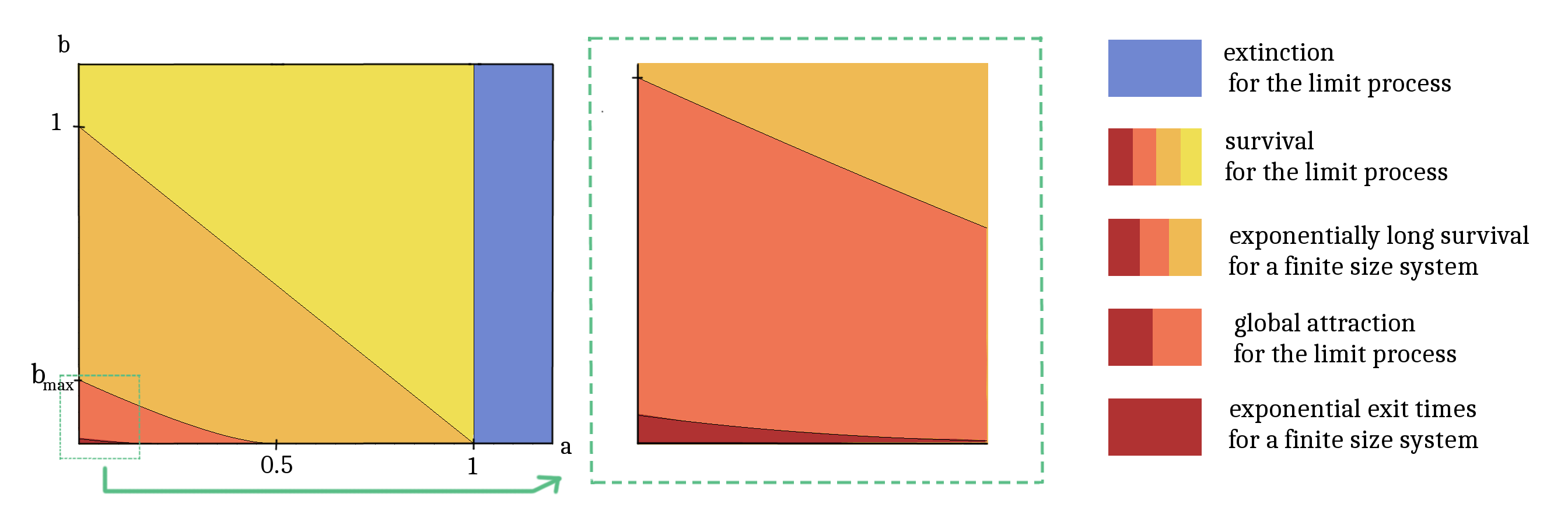}
\end{center}
\caption{Summary of the results when $\lambda(u)= (ku)\wedge \lambda_*$  with $a=\alpha/(kh)$ and $b=\lambda_*/(kh)$. If $a>1$, $\delta_{\bf 0}$ is the unique equilibrium of the limit process, globally attractive (Proposition~\ref{prop:extinction}) and if $a<1$ it is unstable and there exists at least a positive equilibrium (Theorem~\ref{theo:41} and Proposition~\ref{prop:delta0instable1}). If $a+b<1$, the last spike time for a finite system is exponentially large with the size of the system (Theorem~\ref{thm:LDPLeta}). Under the condition \eqref{condition_a_b_y0}, the positive equilibrium of the limit process is unique and globally attractive (Theorem~\ref{theo:gattractive}). Under the condition \eqref{condition_a_b_2}, the exit times of a finite size system from a neighborhood of the limit equilibrium converge to an exponential distribution (Theorem~\ref{theo:exitTimes}).}\label{fig-parametre}
\end{figure}

\subsection*{General notation}
Throughout this paper 
\begin{itemize}
\item $\overline A = A^c $ denotes the complementary of a set $A.$
\item
The supremum norm of any real-valued Borel-measurable function $f $ defined on $\R_+$ will be denoted by $\|f\|_{\infty}= \sup_{ x \in \R_+} |f(x)|.$
\item
For any two integers $ n < m , $ $ \llbracket n, m \rrbracket = \{ k \in \N : n \le k \le m \} .$ 
\item
For two probability measures $\nu_1$ and $\nu_2$ on $(\R_+,\mathcal{B} ( \R_+) )$, the Wasserstein distance of order $1$ between $\nu_1$ and $\nu_2$  is defined as
$$
W_1 (\nu_1,\nu_2)=\inf_{\pi\in\Pi(\nu_1,\nu_2)}\left( \int_{\R_+}\int_{\R_+} |x-y| \pi(dx,dy) \right) ,
$$
where $\pi$ varies over the set $\Pi(\nu_1,\nu_2)$ of all probability measures on the product space $\R_+\times \R_+ $ with marginals $\nu_1$ and $\nu_2$. 
\item
$\pi ( ds, dz) ,  \pi^i ( ds, dz) , i \geq 1, $ denotes an i.i.d. sequence of  Poisson random measures on $ \R_+ \times \R_+, $ having intensity $ d  s dz $ each.
\end{itemize}

\section{The model and main results}\label{Section-modeldef}
We consider for each $ N \geq 1$ the Markov process
$$U^N (t) = (U^N_1 (t), \ldots , U^N_N ( t) )   ,  \, t \geq 0 , $$
taking values in $\R_+^N, $ for some fixed integer $N \geq 1 ,$ solution of the stochastic differential equation 
\begin{multline}\label{eq:EDS_U^N}
 U^N_i  (t)  = U^N_i  (0) - \alpha \int_0^t U^N_i ( s)  ds + \frac{h}{N} \sum_{ j =1, j \neq i }^N \int_{[0, t ] \times \R_+}   \indiq_{\{ z \le \lambda ( U^N_j (s- ) ) \}} \pi^j ( ds , dz ) \\
  -\int_{[0, t ] \times \R_+}  U^N_i (s- )  \indiq_{\{ z \le \lambda ( U^N_i  (s- ) ) \}} \pi^i ( ds , dz ) \,,\qquad 1 \le i \le N.
\end{multline}
The associated generator is given for any smooth test function $ \varphi : \R_+^N \to \R $ by
\begin{equation}\label{eq:generator0}
A \varphi (x ) = \sum_{ i = 1 }^N  \lambda (x_i) \left[ \varphi (x + \Delta_i ( x)  ) - \varphi (x) \right]
- \alpha  \sum_i  \frac{\partial \varphi}{\partial x_i} (x)  x_i   ,
\end{equation}
where
\begin{equation}
(\Delta_i (x))_j =    \left\{
\begin{array}{ll}
\frac{h}N & j \neq i \\
- x_i & j = i 
\end{array}
\right\} , 
\end{equation}
and where $h  > 0 $ and $  \alpha > 0 $ are positive parameters. We assume that 
\begin{ass}\label{ass:1}
$\lambda : \R_+ \to \R_+ $ is bounded, increasing and Lipschitz. Moreover we have $\lambda ( 0) = 0 .$ 
\end{ass}
Under the above assumption, existence and uniqueness of a strong solution of~\eqref{eq:EDS_U^N} are a consequence e.g. from Theorem~IV.9.1 of \cite{ikeda_stochastic_1989}.

In what follows we shall define $ \lambda_* = \| \lambda\|_\infty < \infty  .$ Moreover, let us write $ T_0 = 0 < T_1 < T_2 < \ldots < T_n < \ldots $ for the successive jump times of the process. Since $ \lambda_* < \infty, $ they appear at most at the jump times of a rate $ N \lambda_*-$Poisson process. 
  
Under minimal regularity assumptions on the spiking rate, if we work at a fixed system size $N,$ this process will die out in the long run as shows the following
\begin{theo}\label{prop:dieout}[Theorem 2.3 of Duarte and Ost (2016) \cite{do}] 
If $ \lambda $ is differentiable in $ 0, $ then the system stops spiking almost surely, that is, 
$$ L^N := L := \sup \{ T_n : n \geq 1, T_n < \infty \} < \infty $$
almost surely.  As a consequence, the unique invariant measure of the process $U^N $ is given by $\delta_{\bf 0 }, $
where $\bf 0 \in \R^N $ denotes the all-zero vector in $\R^N .$ 
\end{theo}

This paper is devoted to the study of the large $N$ asymptotics of this last spiking time $ L^N$ and related exit times. 

In addition to Assumption \ref{ass:1} we suppose that 
\begin{ass}\label{ass:2}
$\lambda $ is Lipschitz continuous with  
 $ \lambda ' (u) u \le \frac{r}{\alpha } \lambda ( u) $ for all $u\ge 0$ for some $r>0$. Moreover, there exists $ u_* > 0 $ such that $ \lambda ' (u) \geq k $ for all $ u \le u_*  .$ 
Finally, we have that   
$ k h  > \lambda_* .$
\end{ass}

\begin{rem}
If $\lambda$ is not differentiable, the conditions on $\lambda'$ are to be understood as
\[\lambda(x) - \lambda(y) \geq k(x-y)\,, \forall x,y\, :\, y\le x\le u_*\,,\qquad \lambda(x) - \lambda(y) \le \int_y^x \frac{r\lambda(z)}{\alpha z}dz\,, \forall  x,y\, :\, 0<y\le x\,.\]
This holds for instance if $\lambda$ is concave piecewise $\mathcal C^1$ with the conditions satisfied on each interval where $\lambda'$ is defined.
\end{rem}
 
Introduce 
\begin{equation}\label{eq:lambdaN}
\Lambda^N (t) := \sum_{i=1}^N \lambda (U_i^N(t) )  ,
\end{equation}
the total spiking rate of the system. The next result shows that the last spiking time $L^N$ is exponentially large in $N,$ provided $kh$ is large enough and $ \Lambda^N (0) $ not degenerate.

\begin{theo}\label{thm:LDPLeta}
Grant Assumptions~\ref{ass:1} and \ref{ass:2} and suppose the large deviation principle holds in $ \R$ for $ \Lambda^N ( 0) /N $ with good rate function.  Assume moreover that $kh>\lambda_*+r$ and that the law of the initial condition $U^N(0)$ is such that for all $\varepsilon>0$, there exist $x_0>0$ such that $\mathbb P(\Lambda^N(0) \geqslant N x_0)\geqslant 1-\varepsilon$ for all $N$ large enough.   Then for all $\delta > 0 ,$ 
\begin{equation}\label{eq:bigdev}
   \lim_{N\rightarrow +\infty} \P \po L^N \geqslant  e^{ (W_0 - \delta) N}  \pf\  =\ 1\,,
\end{equation}
where
\[W_0 \ = \  \frac{\lambda(u_*)}{kh} \po  \frac{kh-\lambda_*}{r}  - 1  - \ln\po \frac{kh-\lambda_*}{r} \pf -\frac12 \ln^2\po \frac{kh-\lambda_*}{r} \pf \pf\ > \ 0\, .\]
\end{theo}
This is proven in Section~\ref{sec:LDP2}.

\begin{rem}
The other quantities of Assumption~\ref{ass:2} being fixed, note that $W_0\rightarrow +\infty$ as $r\rightarrow 0$ and that $W_0\rightarrow \lambda(u_*)/r$ as $kh\rightarrow +\infty$.
\end{rem}

The further study of $L^N $ and the longtime behavior of $ U^N$ is related to the one of the associated non-linear limit process. More precisely, 
as $N\rightarrow \infty$, the trajectory of a neuron is expected to converge to a process $ \bar U $ solving 
\begin{equation}\label{eq:limitU}
 d \bar U(t) = - \alpha \bar U (t) dt + h \E ( \lambda (\bar U(t)) dt -\bar U (t- )  \int_{\R_+} \indiq_{\{ z \le \lambda (\bar U (t- ) ) \}} \pi ( dt , dz ) ,
\end{equation} 
where $ \pi (dt, dz) $ is a Poisson random measure on $ \R_+ \times \R_+ $ having intensity $ dt dz.$ Under our assumptions, equation \eqref{eq:limitU} is known to be well-posed and to possess a unique strong solution (see Theorem 4 of \cite{fournier} or Theorem 5 of \cite{cormier}).
The convergence of $ U^N $ to $ \bar U$  will be detailed   in Section~\ref{subsec:chaos} (see Proposition~\ref{prop:propchaos}). For now, let us focus on the stability and long-time behaviour of this limit process. We first investigate the invariant states of the limit equation \eqref{eq:limitU}, under general conditions on the jump rate. 
\begin{theo}\label{theo:41}
Assume that $\lambda$ is non-negative, bounded by $\lambda_*$, Lipschitz with $\lambda(0)=0$, and that there exist $u_*>0$ and $k>0$ such that $\lambda(u)\geqslant ku$ for all $u\in[0,u_*]$.  Then, if $kh>\alpha,$ 
 the nonlinear equation \eqref{eq:limitU} has at least two invariant probability measures
supported in $\R_+$. The first is $ \delta_{\bf 0} $. The others are of the form $g(dx)=g(x)dx$,
with $g:[0,\infty) \mapsto [0,\infty)$ given by
\begin{equation}\label{eq:invariantg}
 g (x) = \frac{p_*}{h p_*  - \alpha x } 
\exp \Big( - \int_0^x \frac{\lambda (y) }{h p_* - \alpha  y  } dy \Big) 
\indiq_{\{ 0 \le x < h  p_*/ \alpha  \} }
\end{equation}
for some
\[ p_* \in \left[ \frac{\alpha}h \po u_*\wedge \frac{kh-\alpha}{\|\lambda\|_{Lip}}\pf  ,\lambda_*\right]\]
such that $\int_0^\infty g(dx)= 1$ and $\int_0^\infty \lambda (x) g(dx) = p_*  $. 
\end{theo}

Under the same condition $hk > \alpha$, we can prove that $\delta_{\bf 0}$ is unstable~:

\begin{prop}\label{prop:delta0instable1}
Grant Assumption~\ref{ass:1} and assume that there exist $u_*>0$ and $k>0$ such that $\lambda(u)\geqslant ku$ for all $u\in[0,u_*]$.  Then, if $kh>\alpha$, there exists $c>0$ such that the following holds. For all probability distributions $\mu_0\neq \delta_{\bf 0}$ on $\R_+$, denoting by $\mu_t$ the law at time $t$ of the process \eqref{eq:limitU} with initial distribution  $\mu_0$, there exists $T>0$ such that $z_t := \int_0^\infty \lambda d\mu_t \geqslant c$ for all $t\geqslant T$.
\end{prop}

On the contrary, other parameters lead to the extinction of the system:

\begin{prop}\label{prop:extinction}
Suppose that $\lambda(u) \le k u$ for all $u\ge 0$ for some $k>0$ such that $kh<\alpha$. Then $ \delta_{\bf 0}$ is the only equilibrium of \eqref{eq:limitU}, and it is globally attractive: If $(\eta_t)_{t\ge0}$ is the law of a solution $\bar U$ of \eqref{eq:limitU} with initial condition $\eta_0$, then
\[W_1(\eta_t, \delta_{\bf 0})\ \le\ e^{-(\alpha -kh)t}W_1(\eta_0, \delta_{\bf 0})\,.\]
\end{prop}

Theorem~\ref{theo:41} and Propositions~\ref{prop:delta0instable1}  and \ref{prop:extinction} are proven in Section~\ref{sec:instability-nl}.

\begin{rem}
If $h\lambda'(0)> \alpha$ then Theorem~\ref{theo:41} and Proposition~\ref{prop:delta0instable1} holds with a sufficiently small choice of $u_*$. On the contrary, if $\lambda$ is concave with $h\lambda'(0)<\alpha$ then Proposition~\ref{prop:extinction} holds. If we assume that $\lambda(u) \le k u$ with $k h <\alpha$ only for $u$ smaller than some threshold $u_*$, then the proof of Proposition~\ref{prop:extinction} still applies to any initial condition with support in $[0,u_*]$, so that $\delta_{\bf 0}$ is at least locally stable.
\end{rem}

In the large $kh$ regime, we are able to show that the second invariant measure $g $ obtained in Theorem \ref{theo:41} is unique and globally attracting. For the simplicity of computations, in the sequel, we consider an explicit piecewise linear rate, although  we expect the proofs to work more generally, at least in the case where $\lambda$ is concave,  reaches the value $\lambda_*$ and saturates (without necessarily being linear beforehand).
\begin{ass}\label{ass:lambda}
The jump rate is $ \lambda ( u) = (ku) \wedge \lambda_*$ for all $u\ge 0,$ for some $k,\lambda_*>0$. 
\end{ass}
Under Assumption~\ref{ass:lambda}, set $a:=\alpha/kh$ and $b:=\lambda_*/kh$. Remark that Assumption \ref{ass:lambda} with $kh>\lambda_*$ (i.e. $b<1$) implies Assumptions \ref{ass:1} and \ref{ass:2} with the same $k,\lambda_*, $ with $u_*=\lambda_*/k$ and $r=\alpha.$ Similarly, the condition $kh>\lambda_*+r$ enforced in Theorem~\ref{thm:LDPLeta} reads $a+b<1$. We now strengthen  this quantitative  assumption and suppose that the conditions
\begin{eqnarray}\label{condition_a_b_y0}
2a+b<1 & \qquad\text{and}\qquad &  \frac{b}{1-2a-b}\po 1 + \frac{1}{1-2a-b}\pf \ \leqslant \ y_0
\end{eqnarray}
hold, where $y_0> 0.56$ is the solution of $y e^y = 1$. Note that, for $b=0$, \eqref{condition_a_b_y0} is saturated at $a=1/2$ while, for $a=0$, it is saturated at $b=b_{\max}$ where
\[\frac{b_{\max}}{1 -b_{\max}}\po 1 + \frac{1}{1-b_{\max}}\pf \ = \ y_0 \qquad \Leftrightarrow \qquad b_{\max} \  = \ 1 - \frac{1}{\sqrt{y_0+1}} \ \simeq \ 0.20\,.\]
Moreover, for all $b\in[0,b_{\max}]$, 
\[ \frac{b}{1-2a-b}\po 1 + \frac{1}{1-2a-b}\pf \ = \ y_0 \qquad \Leftrightarrow \qquad  a \ =\ \frac{1-b}{2} - \frac{1}{\sqrt{1+\frac{4y_0}{b}}-1}\,.\]
The set of parameters for which \eqref{condition_a_b_y0} holds is represented in red in Figure~\ref{fig-parametre}. Notice that the condition \eqref{condition_a_b_y0}, meant to be relatively simple and explicit, is the result of several rough bounds in the proofs to make the reading easier~: in any case, we don't expect our arguments, even applied with more care, to give sharp conditions.

\begin{theo}\label{theo:gattractive}
Under Assumption \ref{ass:lambda} and provided \eqref{condition_a_b_y0} holds,  there exists $\kappa\in(0,1)$ with the following property. For all $\gamma>0$, there exists $C_\gamma>0$ such that for all initial conditions with $z_0\wedge \tilde z_0 \geqslant \gamma$, 
\[W_1(\eta_t,\tilde \eta_t) \ \leqslant  \ C_\gamma \kappa^t W_1(\eta_0,\tilde \eta_0)\,.\]
In particular, there exists a unique non-zero equilibrium $g$ and, as soon as $\eta_0\neq \delta_{\bf 0}$, 
\[W_1(\eta_t,g) \ \underset{t\rightarrow +\infty}\longrightarrow \ 0\,.\]
 \end{theo}
Theorem~\ref{theo:gattractive} is proven in Section~\ref{sec:couplage-non-lineaire}. 
 \begin{rem}
 More precisely, what is established in the proof is that, for all $t\geq 0$,
\[ \E\po | U(t)-\tilde U(t)|\pf\ \ \le \ C_\gamma \kappa^t \E\po | U(0)-\tilde U(0)|\pf\,,\]
where $(U,\tilde U)$ is the synchronous coupling of the processes, namely is the coupling  obtained by using the same Poisson noise for both processes, hence making them jump simultaneously as much as possible.
 \end{rem}
 Finally, we come back to the study of the process $U^N$ of interacting neurons . Provided Assumption~ \ref{ass:lambda} and condition~\eqref{condition_a_b_y0} are enforced, let $p_* = \int_0^\infty \lambda g$ where $g$ is the unique positive non-linear equilibrium given by  Theorem~\ref{theo:gattractive}. For $u\in\R_+^N$ and $\delta>0$ we denote $\jump(u)=\sum_{i=1}^N \lambda(u_i)/N$ and $\{\jump \geqslant \delta\}=\{u\in\R_+^N,\ \jump(u) \geqslant \delta\}$ (and similarly for $\{\jump < \delta\}$, etc.).
 
 We strengthen again condition \eqref{condition_a_b_y0}, assuming that
\begin{eqnarray}\label{condition_a_b_2}
\frac{b}{1-2a-b} \exp \po \frac{b}{1-2a-b} \pf \po 1 + \frac{1}{1-2a-b} \exp\po \frac{4+2b}{1-2a-b}\pf \pf & \leqslant & 1\,,
\end{eqnarray}
still denoting $a=\alpha/(kh)$ and $b=\lambda_*/(kh)$.  Denoting $c=1/(1-2a-b)$ and $d=bc$, this condition is saturated when
\[c e^{4c} \ = \ \po \frac1d e^{-d}-1\pf e^{-2d}\,.\]
For a fixed $d>0$, there is no positive solution $c$ to this equation if $d\geqslant y_0$, and a unique one if $d\in(0,y_0)$. Moreover, the conditions $b>1$ and $a>0$ require $d+1<c$, which requires $d<y_1$ where $y_1\simeq 0.016$ is the solution of  
\[(y_1+1) e^{4(y_1+1)} \ = \ \po \frac1{y_1} e^{-y_1}-1\pf e^{-2y_1}\,.\]
The set of parameters for which \eqref{condition_a_b_2} holds is represented in dark red in Figure~\ref{fig-parametre}. Our last main result is concerned with the asymptotic exponentiality of the exit times from some domains. We do not deal directly with the last spiking time $L^N$, but rather with some stopping times that are related to $L^N$, namely times where the average jump rate $\bar \lambda$ becomes small or get far from its corresponding value in the non-linear equilibrium (see Section~\ref{sec:ldp1}, and in particular Remark~\ref{rem:expoL}, for a discussion on $L^N$ and its link with the times where $\bar \lambda$ gets small).
 
 \begin{theo}\label{theo:exitTimes}
Grant Assumption  \ref{ass:lambda} and condition~\eqref{condition_a_b_2}.  Let $\tau = \inf\{t\geqslant 0,\ U^N\notin \mathcal D\}$ where either~:
\begin{enumerate}
\item $\mathcal D = \{\jump \geqslant \gamma\}$ for some $\gamma\in (0,\lambda_*(1-a-b)/2)$.
\item $\mathcal D$ is a measurable subset of $\R_+^N$ such that, for  some $  \delta>0$,
\[\{p_*-\delta \leqslant \jump \leqslant p_*+\delta\} \ \subset \ \mathcal D \ \subset\ \{\jump \geqslant \delta\}. \]
\end{enumerate}
In case $\mathrm{(1)}$ let $\mathcal K =  \{\jump \geqslant \delta\}$ for some $\delta>\gamma$ and in case $\mathrm{(2)}$ let $\mathcal K =  \{p_*-\gamma \leqslant \jump \leqslant p_*+\gamma\} $ for some $\gamma\in(0,\delta)$. Then, in both cases, there exist $C,\theta,N_0$ such that the following holds:
\[\sup_{u\in\R_+^N}\mathbb E_u\po \tau \pf \ < \ \infty\,,\]
and  for all $N\geqslant N_0$,
\[\sup_{t\geqslant 0}\sup_{u\in\mathcal K}\left|\mathbb P_u \po \tau \geqslant t \mathbb E_u\po \tau\pf \pf - e^{-t}\right| \ \leqslant \ \varepsilon(N)   \]
and
\[\sup_{u,v\in \mathcal K}\left|\frac{\mathbb E_u\po \tau\pf }{\mathbb E_v\po \tau\pf }-1\right| \ \leqslant \ \varepsilon(N) \,,\]
where $\varepsilon(N) =C e^{-\theta N}$ in case $\mathrm{(1)}$  and $\varepsilon(N) = C\ln N /N^{1/4}$ in case $\mathrm{(2)}$ .
\end{theo}
Theorem~\ref{theo:exitTimes} is proven in Section~\ref{sec:conclusion-expo}.

To conclude this presentation of our main results, let us notice that Theorem~\ref{theo:exitTimes} is obtained by applying a general result, Theorem~\ref{thm:gene-exit}, which is the topic of Section~\ref{sec:generalresult} and is of independent interest. Theorem~\ref{thm:gene-exit} establishes asymptotic exponentiality for a general Markov process, based on some coupling and hitting times estimates. It is a generalization (and follows the general strategy) of Brassesco, Olivieri and Vares for low-noise diffusion processes \cite{brassesco}.

\section{Exponential bound on the last spiking time}\label{sec:boundL}

The goal of this section is to establish Theorem~\ref{thm:LDPLeta}.

\subsection{An auxiliary Markov process}\label{section:auxiliaryprocess}
We start  by introducing a simple auxiliary Markov process whose large $N$ asymptotics are easy to study. Under minimal additional assumptions on the spiking rate $ \lambda $ and the synaptic weight $h,$ it is possible to compare the process $ \Lambda^N (t) $ introduced in \eqref{eq:lambdaN} with a simple Markov process $Z^N (t)$ that we are going to introduce next. Assumptions~\ref{ass:1} and \ref{ass:2} are enforced.

For all $N>h/u_*$, set 
\[z_N \  = \ \po 1 - \frac{\lambda_*}{kh}-\frac1N\pf \lambda\po u_* - \frac{h}{N}\pf - \frac{\lambda_*}{N}\]
and consider the function $m_N$ on $\R_+$ given by 
\[m_N(z) \ = \   z+ \frac{kh}{N}\po 1 - \frac{z}{\lambda\po u_*-\frac{h}{N}\pf} - \frac{1}{N}\pf_+ - \frac{\lambda_*}{N} \]
where $(x)_+=x\vee 0$. Using that $kh>\lambda_*$, we see that $m_N$ is a positive increasing function for $N$ large enough. Let $N_0>h/u_*$ be large enough so that for all $N\geqslant N_0$, $z_N>0$ and $m_N$ is a positive increasing function. Remark that $z_N$ has been defined so that $z_N+\lambda_*/N$ is the solution of $z=m_N(z)$ and such that for all $0 \le  z < z_N +\lambda_*/N, $ $m_N ( z) > z .$ 

For $N\geqslant N_0$, introduce the auxiliary Markov process  $Z^N $ taking values in $\R_+$ having generator given for all smooth test functions $\varphi$ by 
\begin{equation}\label{eq:zN}
 A^N  \varphi (z) = - r z \varphi' (z) +  N z   \co  \varphi  \po z_N \wedge m_N(z)\pf  - \varphi( z) \cf  .
\end{equation}

\begin{prop}\label{prop:Z} 
Grant Assumptions \ref{ass:1}, \ref{ass:2}, suppose that $N\geqslant N_0$ and $ Z^N (0) \le  (  \Lambda^N (0)/N)\wedge z_N $.
Then there exists a coupling of $Z^N $ and $ \Lambda^N $ such that $ Z^N (t) \le \Lambda^N (t)/ N $ for all $ t \geq 0.$ 
\end{prop}

\begin{proof} 
{\bf Step 1.}  Consider the effect of jumps on the original process $ \Lambda^N (t).$  Since $ \lambda $ is increasing,
\[\Lambda^N(t) \ \geqslant \ \lambda\po u_*-\frac hN\pf \mathrm{card}\left\{ i : U^N_i (t ) >  u_*-\frac h N \right\}\]
and thus
\[\mathrm{card}\left \{ i : U^N_i (t) \leqslant u_*-\frac h N \right\}  \ \geqslant  N - \frac{\Lambda^N(t)}{\lambda\po u_*-\frac hN\pf}  . \]
 By choice of $u_* $, if a spike occurs at time $t$, all the neurons satisfying $U^N_i (t- ) \leqslant u_*-h/N$ (except possibly one of them that is  spiking) 
 receive an increase of their spiking rate given by 
$$ \lambda \po U^N_i (t- )  + \frac{h}{N} \pf - \lambda \po U^N_i (t- ) \pf \ \geq\  \frac{kh}{N} \,.$$
On the other hand,  the spiking particle is reset to $0$ and $ \lambda (0) = 0$, so that its spiking rate is at most decreased by $\lambda_*$. Therefore, if a spike occurs at time $t$, 
\begin{equation}\label{eq:increaselambda}
 \Lambda^N (t) \ \geq \ \Lambda^N ({t-})  - \lambda_*  +   \frac{kh}{N}  \po N - 1 - \frac{\Lambda^N(t-)}{\lambda\po u_*-\frac hN\pf}\pf_+ \ = \  N m_N\po \frac{\Lambda^N(t-)}N\pf .
\end{equation} 
The definition of $z_N$ has been chosen to ensure that
\[\frac{\Lambda^N(t-)}N \leqslant z_N + \frac{\lambda_*}{N} \qquad \Rightarrow \qquad \Lambda^N(t) \geqslant \Lambda^N(t-)\,.\]
Besides, 
\[\frac{\Lambda^N(t-)}N \geqslant z_N + \frac{\lambda_*}{N} \qquad \Rightarrow \qquad \frac{\Lambda^N(t)}N \geqslant z_N\,.\]

{\bf Step 2.}  We couple $ Z^N$ and $ \Lambda^N /N $ by forcing them to jump together as often as possible (for a general discussion on  synchronous couplings for PDMP's, see \cite{DGM}). In other words we let $Z^N$ evolve according to
\[d  Z^N(t) \ =  \ -r Z(t)dt + \po \po z_N\wedge m_N( Z^N(t-))\pf - Z^N(t-)\pf \int_{\R_+} \indiq_{\{z \leqslant  Z^N(t-)\}} \sum_{i=1}^N\pi^i(dt,dz) . \]
 Since $ d \Lambda^N (t) \geq - r \Lambda^N (t) dt $ in between successive jumps of the system, the deterministic flow preserves the stochastic ordering $ Z^N (t) \le \Lambda^N (t)/N$ up to the first jump time. Remark that $Z^N (t)  \in [0,z_N]$ for all $t\ge 0$. Jumps of the neuron system arrive at rate $ \Lambda^N ({t-}), $ and those of the Markov process at rate $NZ^N ({t-}), $ such that a jump of $Z^N$ is necessarily a jump of  the original system  at least up to time $   \tau  = \inf \{ t : Z^N (t) >\Lambda^N (t) / N \} $. One the one hand, if $\Lambda^N$ jumps alone at a time $t$ where $\Lambda^N(t-)/N \geqslant Z^N(t-)$, then either $\Lambda^N(t-)/N \leqslant z_N + \lambda_*/N $ in which case $\Lambda^N(t)\geqslant \Lambda^N(t-) \geqslant N Z^N(t)$, or   $\Lambda^N(t-)/N \geqslant z_N +\lambda_*/N$ in which case $\Lambda^N(t)/N\geqslant z_N \geqslant Z^N(t)$. On the other hand,  if both $\Lambda^N$ and $Z^N$ jump at a time $t$ where $\Lambda^N(t-)/N \geqslant Z^N(t-)$, then 
\[Z^N(t) \ \leqslant \ m_N\po Z^N(t-)\pf \ \leqslant \ m_N\po \frac{\Lambda^N(t-)}{N}\pf \ \leqslant \ \frac{\Lambda^N(t)}{N}\,,\]
where we used that $m_N$ is increasing for $N\geqslant N_0$ together with \eqref{eq:increaselambda}. This proves that the stochastic ordering $ Z^N (t) \le \Lambda^N (t)/N$ is preserved at the jump times, which concludes  our proof (in particular, $\tau  = +\infty$ almost surely).

\end{proof}

\subsection{Lower bounding the last spiking time of the system}\label{sec:ldp1}
Under the conditions of Theorem~\ref{prop:dieout}, the finite system of interacting neurons $U^N $ possesses a last spiking time $ L= L^N.$ This last spiking time is not a stopping time of the process. However we may consider an enlargement of the original process $ U^N$ such that $ L $ becomes indeed a stopping time. For that sake consider the Markov process $ (U^N (t) , E (t) )  \in \R_+^{N + 1 } $  which is defined as follows. We fix an i.i.d. sequence $(\tau_n)_{n \geq 0 } $ of exponential random variables having parameter $1,$ independent of anything else, and we take $ E(0 ) = \tau_0 . $ Up to the first jump time $T_1,$ the process $ U^N = (U^N_1, \ldots, U^N_N)$ evolves according to 
\begin{equation}\label{eq:dynU}
 d U^N_i ( t) = - \alpha U^N_i (t) dt , 1 \le i \le N, 
\end{equation}  and the process $E (t) $ according to 
\begin{equation}\label{eq:dynE}
 d E (t) = - \Lambda^N (t) dt , \mbox{ where as before } \Lambda^N (t) = \sum_{i=1}^N \lambda ( U^N_i (t) ) .
\end{equation} 
We define the first jump time of the process by 
$$ T_1 = \inf \{ t \geq 0 : E ({t-})= 0 \} .$$ 
At time $T_1,$ the process $ U^N ({T_1-})   $ makes a transition $ U^N ({T_1-})   \mapsto U^N ({T_1-})   + \Delta_i (U^N ({T_1-})   ) , 1 \le i \le N, $ with probability $ \frac{\lambda (U^N_i ({T_1-})  }{\Lambda^N ({T_1- }) } .$  
Moreover, we put $ E ({T_1}) := \tau_1 ,$ and start again with the dynamics \eqref{eq:dynU}-\eqref{eq:dynE} up to the next jump 
$$ T_2 = \inf \{ t \geq T_1 : E ({t- }) = 0 \} .$$  
It is evident that the process $ U^N (t)$ follows the same dynamics as the one given in \eqref{eq:generator0}. Moreover, 
\begin{equation}\label{eq:lastspike}
 L = \inf \{ t : E (t) > \int_0^\infty \sum_{i=1}^N \lambda ( e^{- \alpha s } U^N_i (t) ) ds \} 
\end{equation} 
is now a stopping time with respect to the canonical filtration of the enlarged process $ (U^N, E ).$ That being said, in fact we won't use this enlarged process in the following.

\begin{rem}
The above construction is somewhat similar to the one considered in a simpler frame in  \cite{cottrell}. Therein, the deterministic dynamic is given by $ d U^N_i (t) = - dt $ and $ \lambda ( u) = u .$  
\end{rem}

\begin{rem}\label{rem:expoL}
 With the coupling constructed in Proposition \ref{prop:Z}, we obviously have that 
$ L' \le L $ almost surely, where $L' $ is the last jump time of $Z^N .$ Indeed, there is a spike in the system whenever $Z^N$ jumps.

To control the behavior of $ L' $ for large $N, $ it is however easier to consider  
\begin{equation}\label{eq:leta}
  L^N_\eta = \inf \{ t : Z^N (t) \le \eta \} ,
\end{equation}
for some small $\eta > 0$. 
Since the process $Z^N  $  takes values in $[0,z_N]\subset [ 0, z_\infty]$ with $z_\infty = (1-\lambda_*/(kh))\lambda(u_*)$, in the absence of jumps, it needs at most a time 
\begin{equation}\label{eq:S}
 S  = \frac{ \ln (z_\infty) - \ln \eta }{r} 
\end{equation} 
to reach the level $ \eta .$ Therefore, 
$ \{ L^N_\eta > t +S  \} \subset \{ L' > t \} ,$
implying the lower bound 
\begin{equation}\label{eq:lowerboundL}
L^N =  L \geq L' >  L^N_\eta - S  \mbox{ almost surely.} 
\end{equation} 
\end{rem}

In what follows we provide large deviation estimates for $ L^N_\eta $ as $N, $ the number of neurons, tends to infinity. Being interested in $L^N_\eta $ implies that we  only consider the evolution of $Z^N (t)$ for $ t \le L ^N_\eta .$ We may therefore study a slightly different process $ \bar Z^N $ starting from $ \bar Z^N (0) = Z^N (0)\leqslant z_N$ and having generator 
$$  \bar A^N \varphi(z) = - r z\varphi' (z)  + N (\eta \vee z\wedge z_\infty )  \co \varphi \po z\vee z_N\wedge   m_N(z)\pf - \varphi(z) \cf  $$
instead of studying $ Z^N .$ The advantage of considering $ \bar Z^N$ instead of $ Z^N$ is that its jump rate function 
\begin{equation}\label{eq:f}
 f (z) = \eta \vee z \wedge z_\infty 
\end{equation}
is strictly lower bounded,  bounded and Lipschitz continuous. 

\subsection{Large deviations for the auxiliary Markov process}
To study the large deviation principle for the auxiliary process $ \bar Z^N, $ we rely on the theory developed in Chapter 10 of Feng and Kurtz \cite{feng-kurtz}.  
As $N \to \infty, $ 
\begin{eqnarray*}
(x\vee z_N\wedge   m_N(x) ) - x  &=& \frac1N \po kh \po 1 - \frac{x}{\lambda(u_*)}\pf_+ - \lambda_*\pf \indiq_{x<z_\infty} + \underset{N\rightarrow +\infty}  o\po \frac1N\pf\\
 &=& \frac1N \po kh \po 1 - \frac{x}{\lambda(u_*)}\pf- \lambda_*\pf_+ + \underset{N\rightarrow +\infty}  o\po \frac1N\pf ,
\end{eqnarray*} 
which yields the convergence of the generator $\bar A^N \varphi (x) \to (-rx + f(x)G(x))  \varphi' (x) $ with
\[G(x)  \ = \ \po kh \po 1 - \frac{x}{\lambda(u_*)}\pf- \lambda_*\pf_+\,. \]
 The associated dynamics of the limit process is given by 
\begin{equation}\label{eq:limitz}
\dot x_t = - r x_t  + G(x_t) f(x_t)  .
\end{equation}
To quantify the convergence of $ \bar Z^N $ to this limit trajectory $x_t ,$ we consider the associated exponential semigroup 
$$ H^N \varphi(x ) = \frac{1}{N} e^{-N \varphi } \bar A^N e^{N \varphi } (x) = - r x  \varphi' (x)  + f(x)  ( e^{N ( \varphi (x\vee z_N\wedge   m_N(x)) - \varphi (x))  } - 1 ) ,$$
which converges, as $ N \to \infty , $ to 
$$ H \varphi (x) :=  - r x  \varphi' (x)  + f(x)  ( e^{ G (x) \varphi'  (x) }- 1) . 
$$
Notice that, for all $ x \geq z_\infty$,  $ G(x) = 0 $ and thus $ H \varphi (x) = - r x \varphi'(x) $. We define for any $p \in \R $
$$ H(x, p ) := f(x)   ( e^{ G (x) p }- 1)  - r x p \quad  \mbox{  and  } \quad 
 L(x, q) = \sup  \{ p  q - H(x, p ) , p \in \R \} .$$ 
Clearly, $ L ( x, q) = + \infty $ if $ q < - r x .$ Moreover, for all $x$ such that $ 0 \le x <  z_\infty  $ and all $ q \geq -r x , $ 
$$ L(x, q)  = \frac{q +r x}{G(x)  } \ln \left( \frac{q +r x}{G (x) f(x) } \right) - \frac{q +r  x}{G(x)} + f(x) = f(x)   \left\{ u \ln u - u + 1 \right\} ,$$
where 
$$ u = \frac{q +r x}{G (x) f(x) }   $$
and $u\ln u =0$ if $u=0$. Moreover, if $ x \ge  z_\infty, $ then 
$$ L ( x, q ) = \left\{ 
\begin{array}{ll}
0 & \mbox{ if } q = - r x \\
+ \infty & \mbox{ if } q \neq  - r x 
\end{array}
\right\} .$$ 
The solution  $ x_t $  of the limit equation \eqref{eq:limitz} satisfies
$$ L ( x, \dot x ) = 0,$$
since in this case $u= 1.$

\begin{theo}[Theorem 10.22 of \cite{feng-kurtz}]\label{theo:ldp}
Suppose that the large deviation principle holds in $\R$ for $  Z^N (0) = \bar Z^N (0) $ with good rate function $ I_0$ and grant Assumptions \ref{ass:1} and \ref{ass:2}. 
Then the large deviation principle holds for $\bar Z^N$ in $ D( \R_+, \R_+) $ with good rate function 
$$ I ( x) = \left\{
\begin{array}{ll} 
I_0 ( x_0 )  + \int_0^\infty L ( x_s , \dot x_s ) ds , & \mbox{ if $x$ is absolutely continuous}\\
+\infty , & \mbox{ else}
\end{array}
\right\} .$$
In particular, for any open set $ A \in \D ( \R_+, \R_+ ), $ 
$$ \liminf_{N \to \infty} \frac1N \log \P ( \bar Z^N \in A ) \geq - \inf_{ x \in A} I(x) ,$$
and for any closed set $ B \in \D ( \R_+, \R_+), $
$$ \liminf_{N \to \infty} \frac1N \log \P (\bar Z^N \in B ) \le - \inf_{ x \in B} I(x) .$$
\end{theo}

\begin{proof}
We are in the framework of Chapter 10 of \cite{feng-kurtz} with $ \eta (x, dz) = f(x) \delta_{G(x) } (dz), b (x) = - r x + f(x) G(x)  $ and $ \psi ( x) = 1 + |x|.$ It is immediate to check that Condition 10.3 is satisfied and that Lemma 10.4 and Lemma 10.12 hold, since $f$ is lower bounded and since $G$ is Lipschitz and bounded. As a consequence, all conditions required to apply Theorem 10.17 and Theorem 10.22 of \cite{feng-kurtz} are met which concludes the proof. 
\end{proof}

\subsection{Stability of the limit system}\label{sec:LDP2}
We briefly discuss the stability properties of the limit system \eqref{eq:limitz}.  In this section we strengthen Assumption \ref{ass:2} by assuming that $kh > \lambda_*+r$. We also chose $\eta$ small enough, more precisely $\eta<x_\infty$ with
\[x_\infty \ :=\ \lambda(u_*) \po 1 - \frac{\lambda_*+r}{kh}\pf \ \in \ (0,z_\infty)\,.\]
In that case $x_\infty$ is the unique equilibrium of the limit equation \eqref{eq:limitz}, and it is globally attracting on $(0,+\infty)$.  Recall we want to study the exit time $ L^N_\eta $ defined in \eqref{eq:leta}.  To do so, one classically introduces the cost functionals 
$$ V_t (x, y ) = \inf_{  x : x_0  = x, x_t = y }  I_t ( x )  , \mbox{ 
where } I_t  (x ) = \int_0^t L (x_s, \dot x_s )   ds\,, \; \mbox{  and  }  \; 
 V (x, y ) = \inf_{t\geqslant 0} V_t ( x, y) . 
$$
Then we have 
\begin{prop}\label{prop:ldp1}
Grant the conditions of Theorem \ref{theo:ldp}, assume moreover that $kh>\lambda_*+r$ and $\eta<x_\infty$, and fix $ x>  \eta$.  
Then,   denoting $ \bar V_\eta = V ( x_\infty , \eta) $,
\begin{equation}\label{eq:result1}
\forall \delta>0\,,\qquad  \lim_N \P_x ( e^{ (\bar V_\eta - \delta) N} < L^N_\eta  < e^{(\bar V_\eta + \delta ) N } ) = 1 
\end{equation}
and
\begin{equation}\label{eq:result2}
 \frac1N \log \E_x \po L_\eta^N\pf \  = \ \bar V_\eta\,  .
\end{equation} 
Moreover,
\[+\infty \ > \ \frac{x_\infty - \eta }{r} \ \geqslant \ \bar V_\eta \ \geqslant \  \frac1r \int_{\eta}^{x_\infty} Q\po \frac{r}{G(z)}\pf   dz \ >\ 0\,,\]
where $Q(u):=u \ln u - u + 1 $.
\end{prop}

\begin{proof}
To get the upper bound on $\bar V_\eta$, choose  $ x_t = e^{- rt } x_\infty , t \geq 0 .$ The trajectory reaches $\eta$ in a finite time $t_\eta = \ln(x_\infty/\eta)/r$ and 
\[\int_0^{t_\eta} L(x_s,\dot x_s) ds \ = \ \int_0^{t_\eta} x_s ds \ = \ \frac{x_\infty - \eta}r \,.\]

To get the lower bound, let us fix some time horizon $t$ and take any absolutely continuous trajectory $ x : [0, t ] \to \R_+ $ such that $ x_0 = x_\infty ,  x_t = \eta $ and  $ x_s =  x_\infty  + \int_0^s \dot x_u du , $ for all $ 0 \le s \le t.$ Since $L(z,q)\geqslant 0$ for all $z,q$, it is easy to show that the cheapest way to go from $ x_\infty$ to $ \eta $ during the time interval $ [0,t]$ is to have negative derivative at all times, that is, $ \dot x_s \leqslant 0 $ for almost all $ s \in [0,t].$ Besides, in order to have a finite cost, necessarily $ \dot x_s \geqslant -r x_s $ for almost all $ s \in [0,t].$ Hence, in the following we suppose that $-rx_s\leqslant\dot x_s \leqslant 0$ for almost all $s\in[0,t]$.

Now, for all $ s \in [0, t ],$  
$$ u_s := \frac{ \dot x_s + rx_s}{ G(x_s) x_s }\ \in \ \left[0,\frac{r}{G(x_s)}\right] \subset [0,1]\,,$$
where we used that $G$ is decreasing with $G(x_\infty)=r$. Since $Q$ is decreasing on $ [0, 1], $ this implies that
\[L(x_s,\dot x_s) \ = \ x_s Q(u_s) \ \geqslant \ x_s Q\po \frac{r}{G(x_s)}\pf \ \geqslant \ - \frac{\dot x_s}{r}Q\po \frac{r}{G(x_s)}\pf\,. \]
As a consequence,
\[\int_0^t L(x_s,\dot x_s) ds \ \geqslant \ -\frac1r \int_0^t \frac{d}{ds}\po \int_0^{x_s} Q\po \frac{r}{G(y)}\pf dy\pf ds \ = \ -\frac1r\int_{x_\infty}^{\eta} Q\po \frac{r}{G(y)}\pf dy\,.\]
The proof of the lower bound of $\bar V_\eta$ is then concluded by noticing that we have obtained a lower bound that is independent from $t$.

The proof of \eqref{eq:result1} and \eqref{eq:result2} is then similar to the proof of Theorem 7.8 in \cite{pardoux}.
\end{proof}

Denote
\[W_0 \ :=\ \frac1r\int_0^{x_\infty} Q\po \frac{r}{G(y)}\pf dy\,.\]
Remark that a consequence of Proposition~\ref{prop:ldp1} is that $W_0\leqslant x_\infty/r <+\infty$. More precisely, as $x_\infty \leqslant z_\infty$, $G$ is affine on $[0,x_\infty]$, and the change of variable $z=G(y)/r$  yields
\begin{eqnarray*}
W_0 &=& \frac{\lambda(u_*)}{kh}\int_1^{\frac{kh-\lambda_*}{r}} Q\po \frac1z\pf dz\\
& = & \frac{\lambda(u_*)}{kh}\co -\frac12 \ln^2(u) - \ln (u)  + u\cf_1^{\frac{kh-\lambda_*}{r}} \\
& = & \frac{\lambda(u_*)}{kh} \po  \frac{kh-\lambda_*}{r}  - 1  - \ln\po \frac{kh-\lambda_*}{r} \pf -\frac12 \ln^2\po \frac{kh-\lambda_*}{r} \pf \pf\, .
\end{eqnarray*}
As a conclusion of this section, recalling that the last spiking time $ L = L^N $ has been introduced in Theorem \ref{prop:dieout}~:

\begin{proof}[Proof of Theorem~\ref{thm:LDPLeta}]
Let $\varepsilon,\delta>0$, and let $x_0\in (0,x_\infty)$ be such that $\mathbb P(\Lambda^N(0) \geqslant N x_0)\geqslant 1-\varepsilon$ for all $N$ large enough. Let $Z^N$ be the Markov process with generator \eqref{eq:zN} and $Z^N(0)=x_0$. Let $\eta>0$ be small enough  so that
\[W_0 \leqslant \frac\delta 2 + \frac1r \int_{\eta}^{x_\infty}Q\po \frac r{G(z)}\pf   dz \,.\]
Considering $S$ given by \eqref{eq:S} and using Proposition~\ref{prop:ldp1},
\begin{eqnarray*}
\mathbb P \po L^N \geqslant e^{(W_0-\delta)N} \pf & \geqslant & \mathbb P\po L_\eta^N -S > e^{(W_0-\delta)N} \pf -\mathbb P \po \Lambda^N(0) < N x_0\pf \\
& \geqslant & \mathbb P\po L_\eta^N   > S+ e^{(\bar V_\eta-\delta/2)N} \pf - \varepsilon\\
& \underset{N\rightarrow+\infty}\longrightarrow & 1 - \varepsilon\,.
\end{eqnarray*}
This ends the proof since $\varepsilon$ is arbitrary.
\end{proof}


\section{Longtime behavior of the limit process}\label{sec:4}

\subsection{(In)stability of zero} 
\label{sec:instability-nl}


First we consider the question of the existence of non-zero equilibrium, under general conditions on the jump rate.

\begin{proof}[Proof of Theorem~\ref{theo:41}]
The proof follows the ideas of the proof of Theorem 8 in \cite{fournier} and of Section 6 of \cite{robert-touboul}.  Fix some parameter $ a> 0 .$ Then the  $\R_+$-valued SDE
\begin{equation}\label{La}
Z(t)=Z (0) - \int_0^t \int_0^\infty Z (s-)  \indiq_{\{z\leq \lambda (Z ({s-}) )\}}\pi (ds,dz) + \int_0^t (a-\alpha  Z (s) )ds
\end{equation}
has the unique invariant probability measure $g_a$ given by  
$$ 
g_a (x) = \frac{p_a}{a - \alpha  x } \exp \Big( - \int_0^x \frac{\lambda(y) }{a - \alpha y} dy \Big) 
\indiq_{\{ 0 \le x < a/\alpha  \} },$$
where  $p_a>0$ is such that 
$ \int_0^\infty g_a(x) dx = 1$ (see e.g. Proposition 26 in \cite{cormier}). It automatically holds that $\int_0^\infty \lambda (x)g_a(dx)=p_a$.  When $a=0$, $g_0= \delta_{\bf 0}$ is invariant for \eqref{La} and $p_0:=\int_0^\infty \lambda (x)g_0(dx) = 0$.

Denoting $A_a$ the generator associated to \eqref{La} and $\varphi(x)=x$, the invariance of $g_a$ implies that
\begin{eqnarray*}
0 \ = \ \int_0^\infty A_a\varphi (x) g_a (dx)  & = & \int_0^\infty \po (a-\alpha x)-x\lambda(x) \pf g_a(dx)\\
& \geq & -(\alpha+a\|\lambda\|_{Lip}/\alpha)\int_0^\infty xg_a(dx) +a \,,  
\end{eqnarray*}
where we used that the support of $g_a$ is $[0,a/\alpha]$. In other words,
\[\int_0^\infty xg_a(dx)  \ \geqslant \ \frac{a}{(\alpha+a\|\lambda\|_{Lip}/\alpha)}\,.\]
Moreover, if $a\leqslant \alpha u_*$,
\begin{eqnarray*}
p_a \ = \ \int_0^\infty \lambda(x) g_a(dx)   & \geqslant & k\int_0^\infty x g_a(dx)  \ \geqslant \ \frac{ka}{(\alpha+(a\|\lambda\|_{Lip}/\alpha)\wedge \lambda_*)}\,.
\end{eqnarray*}
Since $hk>\alpha$, then $hp_a> a$ for $a < a_0 :=  \alpha (u_*\wedge (kh-\alpha)/\|\lambda\|_{Lip})$. On the other hand, since $p_a<\lambda_*$ for all $a\geq 0$, $hp_a< a$ for $a\geq h\lambda_*$. Denoting
$$ 
\Gamma (a) \ :=\ \frac{1}{p_a} \ =\  \int_0^{ a / \alpha } \frac{1}{ a - \alpha x } 
\exp \Big(- \int_0^x \frac{\lambda (y) }{a - \alpha y } dy  \Big) dx \,,
$$
we transform the above integral by replacing successively $y$ by $ y / a $ and  $x $ by $ x/a $ to obtain
$$ \Gamma (a) = \int_0^{1/\alpha} \frac{1}{ 1 - \alpha x} \exp \Big( - \int_0^x  \frac{\lambda ( a y ) }{ 1 - \alpha y } dy \Big) dx .$$ 
Clearly,  $a \mapsto \Gamma ( a) $ is continuous. As a consequence, the equation $hp_a = a$ admits at least a solution  $a\in [a_0,h\lambda_*]$ (and no non-zero solution outside $[a_0,h\lambda_*]$). 
\end{proof}

Next, we study the instability of  $\delta_{\bf 0}$ when $kh>\alpha$:

\begin{proof}[Proof of Proposition~\ref{prop:delta0instable1}]

\textbf{Step 1.} 
First we give a rough lower bound of $z_t$ in term of $z_0$. If $\bar U(0) \leqslant u_*$ then, before the first jump,
\[\lambda \po \bar U(t)\pf \ \geqslant \ \lambda\po e^{-\alpha t} \bar U(0)\pf  \ \geqslant \ k e^{-\alpha t} \bar U(0)   \ \geqslant \  \frac{k}{\|\lambda\|_{Lip}} e^{-\alpha t} \lambda(\po \bar U(0)\pf\,. \]
Alternatively, if $\bar U(0) > u_*$ then, before the first jump,
\[ \lambda \po \bar U(t)\pf \ \geqslant \ \lambda \po e^{-\alpha t} u_* \pf \ \geqslant \ k e^{-\alpha t} u_* \ \geqslant \  \frac{k e^{-\alpha t} u_* }{\lambda_*} \lambda\po \bar U(0)\pf\,.\]
Since the first jump arrives at rate less than $\lambda_*$,
\[z_t \ \geqslant \ k e^{-(\alpha+\lambda_*)t} \po \frac1{\|\lambda\|_{Lip}} \wedge \frac{u_*}{\lambda_*} \pf  z_0  \ :=\ c_t z_0\,.\]

\textbf{Step 2.} Assume that 
\[0 \ <\ z_0 \ \leqslant\  \gamma \ :=\ \frac{\alpha}{h} \po u_*\wedge \frac{kh-\alpha}{2\|\lambda\|_{Lip}}\pf\,.\]
 Let $t_0=\inf\{t\geqslant 0,\ z_t>\gamma\}$.   In that case, for all  $s \leqslant t\leqslant t_0$,
\begin{equation}\label{eq:ut}
d \bar U(t) \ \leqslant \ \po-\alpha \bar U(t) + \gamma h\pf dt \qquad \text{and thus  }\qquad \bar U(t) \ \leqslant \ e^{-\alpha (t-s)}\bar U(s) + (1-e^{-\alpha (t-s)}) \frac{\gamma h}{\alpha} \, .
\end{equation}
In particular  $\bar U(s)\leqslant \gamma h /\alpha$ implies that $ U(t) \leqslant \gamma h /\alpha$ holds for all $s \leqslant t\leqslant t_0$. Besides, for all $u\in[0,\gamma h/\alpha]\subset[0,u_*]$,
\[-\alpha u +h \lambda(u) - u \lambda(u) \ \geqslant \ -\alpha u +h k u - u \lambda(u) \ \geqslant \ \po kh - \alpha -\lambda(\gamma h/\alpha)\pf u \ \geqslant \ \frac{kh-\alpha}{2}u\]
by choice of $\gamma$. As a consequence, using that
\[z_t \ = \ \mathbb E \left ( \lambda\po \bar U(t)\pf \right) \ \geqslant \  \mathbb E \left ( \lambda\po \bar U(t)\pf \indiq_{\{\bar U(s)\leqslant\gamma h /\alpha\}} \right) \ \ \geqslant \ k \mathbb E \left ( \bar U(t) \indiq_{\{\bar U(s)\leqslant\gamma h /\alpha\}} \right) , \]
we obtain that, for all $s\leqslant t\leqslant t_0$,
\begin{eqnarray*}
\partial_t \mathbb E\left(\bar U(t) \indiq_{\{\bar U(s)\leqslant\gamma h /\alpha\}} \right) & = & \mathbb E \left( (-\alpha \bar U(t) + hz_t - \bar U(t)\lambda(\bar U(t)))\indiq_{\{\bar U(s)\leqslant\gamma h /\alpha\}} \right) \\
& \geqslant & \frac{kh-\alpha}{2}\mathbb E\left(\bar U(t) \indiq_{\{\bar U(s)\leqslant\gamma h /\alpha\}} \right)
\end{eqnarray*} 
and thus 
\[z_{t} \ \geqslant \ k e^{\frac{kh-\alpha}{2} (t-s)} \mathbb E\left(\bar U(s) \indiq_{\{\bar U(s)\leqslant\gamma h /\alpha\}} \right)\,. \]
 
 \textbf{Step 3.} Suppose that $t_0 > 1$. Using the bound obtained in Step 1, we are going to apply the previous inequality with $s=1$. If $\bar U(0) \leqslant \gamma h /\alpha$ then before  $s=1$ and the first jump, $\bar U(t)$ is in $[0,\gamma h/\alpha]$ and lower bounded by the solution of 
\[\dot x \ = \ -\alpha x  + h c_1 z_0,  \qquad x(0)=0\,.\]
If $\bar U(0) > \gamma h /\alpha$, consider the event where there is a jump in the interval $[0,1/2]$ and no jump in the time interval $(1/2,1]$. In that case $\bar U(t) \in [0,\gamma h/\alpha]$ for $t\in(1/2,1]$ and is lower bounded by the solution of
\[\dot x \ = \ -\alpha x  + h c_1 z_0, \qquad x(1/2)=0\,,\]
so that $\bar U(1) \geqslant h c_1 z_0 (1- e^{-\alpha/2})/\alpha$. Moreover, if $\bar U(0) > \gamma h /\alpha$, then before the first jump $\lambda(\bar U(t)) \geqslant k  e^{-\alpha/2}  \gamma h /\alpha  := \tilde \lambda$. 

So, distinguishing the two cases whether $\bar U(0) $ is greater or less than $\gamma h /\alpha$ we get
\[\mathbb E\left(\bar  U(1) \indiq_{\{U(1)\leqslant\gamma h /\alpha\}} \right) \ \geqslant \ e^{-\lambda_*}\po 1 - e^{-\tilde \lambda/2}\pf \frac{h c_1 z_0}\alpha (1- e^{-\alpha/2}) \ := \ \tilde c z_0\,. \]
Combined with the result of Step 2, we have thus obtained that, if $t_0>1$, then for all $t \in [1,t_0]$, 
\[z_{t} \ \geqslant \ k e^{\frac{kh-\alpha}{2} (t-1)} \tilde c z_0,  \]
and in particular $t_0$ is finite (bounded by a constant that depends only on $\alpha,h,\lambda$ and $z_0$).

\textbf{Step 4.} Assume that $z_0=\gamma$ and that  $t_1=\inf\{t\geqslant 0,\ z_t<\gamma\}$ is finite. Let $t_2 = \inf\{t\geqslant t_1,\ z_t>\gamma\}$. According to the previous step, if $t_2-t_1>1$ then for all $t\in [t_1+1,t_2]$,
\[z_{t} \ \geqslant \ k e^{\frac{kh-\alpha}{2} (t-1-t_1)} \tilde c \gamma , \]
which means that $t_2-t_1$ is bounded by a constant $t_3$ that depends only on $\alpha,h$ and $\lambda$. 

\textbf{Conclusion.} Starting from any initial distribution, after some time, from Step 3, $z_t\geqslant \gamma$. After that time, from Steps 1 and 4, $z_t$ cannot go below the level $\gamma c_{t_3\vee 1}$, which concludes.
\end{proof}
 
To conclude this section,  we prove the stability of  $\delta_{\bf 0}$ when $kh<\alpha$:

\begin{proof}[Proof of Proposition~\ref{prop:extinction}]
Since
\[W_1(\eta_t, \delta_{\bf 0}) \ = \ \E\po |\bar U(t) -0 |\pf\ := \ n_t\,,\]
the conclusion follows from
\[\partial_t n_t \ = \ -\alpha n_t + h \E\co \lambda(\bar U(t))\cf - \E\co \bar U(t) \lambda(\bar U(t))\cf \ \le \ (-\alpha+kh)n_t\,. \]
\end{proof}

\subsection{Propagation of chaos}\label{subsec:chaos}
From now on, for simplicity, we work with piecewise linear jump rates of the form given by Assumption~\ref{ass:lambda}. We first construct under this assumption  an efficient coupling of the finite particle system with the limit process.

\begin{prop}\label{prop:propchaos}
For $N\in\mathbb N_*$, let $\nu_1,\dots, \nu_N$ be probability distributions on $\R_+$ and $\mu_0  = 1/N\sum_{i=1}^N \nu_i$. Consider $U^N$ the system \eqref{eq:EDS_U^N} with initial conditions $U_1^N(0),\dots,U_N^N(0)$ independent with $U_i^N(0)$ distributed according to $\nu_i$ for all $i\in\cco 1,N\ccf$. Let $\mu_t$ be the law of the solution of the limit equation \eqref{eq:limitU}  with initial distribution $\mu_0$, and $z_t = \int_{\R_+}\lambda d\mu_t$.  Consider the process $\bar U^N$ with initial condition $\bar U^N(0)=U^N(0)$ that solves
\[ d \bar U_i^N(t) = - \alpha \bar U_i^N (t) dt + h z_t dt -\bar U_i^N (t- )  \int_{\R_+} \indiq_{\{ z \le \lambda (\bar U_i^N (t- ) ) \}} \pi^i ( dt , dz ) ,
\]
with the same Poisson measures as $U^N$.

Under Assumption \ref{ass:lambda}, for all $t\geqslant 0$, 
$$ \E \po \sum_{i=1}^N | \bar U^N_i (t) - U^N_i (t) | \pf \ \leqslant \ h \po   \sqrt{\lambda_* t } + 2t   \lambda_* \pf    e^{( \alpha + hk + \lambda_*  )t} \sqrt N $$
and
\[\mathbb E \po \left| z_t - \frac1N \sum_{i=1}^N \lambda \po U_i^N(t)\pf\right|\pf \ \leqslant \ \frac{1}{\sqrt N} \po \lambda_* + kh \po    \sqrt{\lambda_* t } + 2t   \lambda_*    \pf  e^{( \alpha + hk + \lambda_*  )t}\pf\,. \]
Moreover, for all $t\geqslant 0$ and all $\varepsilon>0$,
\[\mathbb P \po \sup_{s\in[0,t]} \sum_{i=1}^N |U_i^N(s)-\bar U_i^N(s)|  \geqslant \varepsilon \pf \ \leqslant \ 4 h      ( 1+\sqrt{\lambda_* t } + t   \lambda_* )^2   e^{(2 \alpha + hk + \lambda_*  )t}\frac{ \sqrt N}\varepsilon\,,\]
and there exists $C>0$ (that depends explicitly on $t,\varepsilon,k,h,\lambda_*,\alpha$) such that 
\[\mathbb P \po \sup_{s\in[0,t]} \left| z_s - \frac1N \sum_{i=1}^N \lambda \po U_i^N(s)\pf\right| \ \geqslant \ \varepsilon \pf \ \leqslant \ \frac C {\sqrt N}\,.\]
\end{prop}

\begin{proof}

Let $I$ be a random variable uniformly distributed over $\cco 1,N\ccf$, independent from $(\pi^i)_{i\in\cco 1,N\ccf}$. Then $\bar U_I^N$ solves
\[ d \bar U_I^N(t) = - \alpha \bar U_I^N (t) dt + h z_t dt -\bar U_I^N (t- )  \int_{\R_+} \indiq_{\{ z \le \lambda (\bar U_I^N (t- ) ) \}} \pi^I ( dt , dz ) \,.
\]
Since $\pi^I$ has the same law as $\pi$ and $\bar U_I(0)$ is distributed according to $\mu_0$, this means $\bar U_I(t)$ is distributed according to $\mu_t$ for all $t\geqslant 0$, and in particular 
\[z_t \ = \ \mathbb E \po \lambda \po \bar U_I^N(t)\pf\pf \ = \ \frac1N\sum_{i=1}^N \mathbb E\po \lambda \po \bar U_i^N(t)\pf\pf \,.\]

Notice that a synchronous spike of the pair $(U_i^N,\bar U_i^N)$ for some $i\in\cco 1,N\ccf$ decreases  $|\bar U^N_i (t) - U^N_i (t)|, $ since both of them are reset to $0.$ Asynchronous spikes can only happen if one of the two potentials is below the threshold value $ \lambda_* /k  , $ such that these jumps lead to an increase of at most $\lambda_* / k .$  Therefore, for all $i\in\cco 1,N\ccf$ and $t\geqslant 0$,
\begin{multline*}
|U_i^N(t)-\bar U_i^N(t)|  \ \leqslant \ \alpha\int_0^t |U_i^N(s)-\bar U_i^N(s)|ds + \frac{\lambda_*}{k}\int_{[0,t]\times\R_+}|\indiq_{\{z\leqslant \lambda(U_i^N(s))\}} - \indiq_{\{z\leqslant \lambda(\bar U_i^N(s))\}} |\pi^i(ds,dz)\\
+ h \left|\int_0^t\frac1N\sum_{j\neq i} \int_{\R_+} \indiq_{\{z\leqslant \lambda(U_j^N(s))\}} \pi^j(dz,ds) - \int_0^t z_s ds\right| \,.
\end{multline*}
We bound the last term by
\[h \left|\int_0^t\frac1N\sum_{j=1}^N  \int_{\R_+} \indiq_{\{z\leqslant \lambda(U_j^N(s))\}} \pi^j(dz,ds) - \int_0^t z_s ds\right|  +  \frac{h}N \int_{[0,t]\times \R_+} \indiq_{\{z\leqslant \lambda(U_i^N(s))\}} \pi^i (dz,ds)\,.\]
Therefore, writing 
\begin{eqnarray*}
  M^N_t & = &  \sum_{j=1}^N \int_0^t \int_{\R_+} \indiq_{\{ z \le z_s\}}  \pi^j (ds, dz ) - N\int_0^t z_s ds , \\
 R^N_t &= &   \int_0^t \left| d \int_{\R+}  \sum_{j=1}^N \po\indiq_{\{ z \le \lambda ( \bar U^N_j (s-) ) \}} - \indiq_{\{ z \le z_s \}}\pf   \pi^j (ds , dz )\right| 
 \, ,
\end{eqnarray*}
which is the total variation norm of the signed measure $ \int_{\R+}  \sum_{j=1}^N \po\indiq_{\{ z \le \lambda ( \bar U^N_j (s-) ) \}} - \indiq_{\{ z \le z_s \}}\pf   \pi^j (ds , dz ) $ on $ [0, t ], $  
\[  f(t)  = \sum_{j=1}^N \mathbb E \po |U_j^N(t)-\bar U_j^N(t)|  \pf , \]
and using that $\lambda$ is $k$-Lipschitz, we obtain 
\[
f(t) \le  ( \alpha + hk + \lambda_*  ) \int_0^t   f(s) ds + h \E \po  | M^N_t| \pf + h \E \po R^N_t\pf  +   \lambda_*ht\,.\]
Observe that $\E \po |M_t^N|\pf \leqslant  \sqrt{\lambda_* t N} $ and
\[  \E \po R_t^N \pf \  = \ \int_0^t \mathbb E \po \left|\sum_{j=1}^N \Big(  \lambda ( \bar U^N_j (s)) - \E \co  \lambda ( \bar U^N_j (s))\cf \Big)\right|\pf ds  \ \leqslant \  t   \lambda_* \sqrt N , \]
where we used that the  $\lambda (\bar U_i^N(t))$'s for $i\in\cco 1,N\ccf$ are independent. Then Gronwall's inequality implies the first result. Moreover, 
\begin{eqnarray*}
\mathbb E \po \left| z_t - \frac1N \sum_{i=1}^N \lambda \po U_i^N(t)\pf\right|\pf & \leqslant & \mathbb E \po \left|  \frac1N \sum_{i=1}^N \Big( \lambda \po \bar U_i^N(t)\pf-\E \co \lambda \po \bar U_i^N(t)\pf\cf\Big) \right|\pf + \frac{k}N f(t) \\
& \leqslant & \frac{\lambda_*}{\sqrt N}+ \frac{k}N f(t) \, ,
\end{eqnarray*}
which implies the second item. Finally, if we use Gronwall's inequality before taking the expectation, we get
\[ \sup_{s\in[0,t]} \sum_{i=1}^N |U_i^N(s)-\bar U_i^N(s)|  \ \leqslant \  e^{\alpha t} \po S_t^N + h \sup_{s\in[0,t]} |M_s^N|\pf \]
with
\[S_t^N \ :=\ h R_t^N + \sum_{i=1}^N     \int_{[0,t]\times\R_+} \po \frac{h}N \indiq_{\{z\leqslant \lambda(U_i^N(s))\}} + \frac{\lambda_*}{k}|\indiq_{\{z\leqslant \lambda(U_i^N(s))\}} - \indiq_{\{z\leqslant \lambda(\bar U_i^N(s))\}} |\pf\pi^i(ds,dz)   \,.\]
On the one hand, 
\begin{eqnarray*}
\E\po S_t^N\pf  &  \leqslant & h t   \lambda_* \sqrt N + h t\lambda_* + \lambda_*  \int_0^t \E\po \sum_{i=1}^N |U_i^N(s)-\bar U_i^N(s)| \pf ds \\
& \leqslant & 2 h t  \lambda_*     ( 1+\sqrt{\lambda_* t } + t   \lambda_* )   e^{( \alpha + hk + \lambda_*  )t} \sqrt N\,,
\end{eqnarray*}
and thus, for all $\varepsilon>0$,
\[\mathbb P \po S_t^N \geqslant \varepsilon \pf \ \leqslant \ \frac{1}{\varepsilon}2 h t  \lambda_*     ( 1+\sqrt{\lambda_* t } + t   \lambda_* )   e^{( \alpha + hk + \lambda_*  )t} \sqrt N\,.\]
On the other hand, for all $\varepsilon>0$, applying Doob's maximal inequality,
\[\mathbb P \po h \sup_{s\in[0,t]} |M_s^N|  \geqslant  \varepsilon \pf \ \leqslant \ \frac h\varepsilon\E\po |M^N_t|\pf \ \leqslant \ \frac{h\sqrt{\lambda_* t N}}{\varepsilon}\,.\]
Summing these two inequalities with $\varepsilon$ replaced by $e^{-\alpha t}\varepsilon/2$ concludes the proof of the third claim.

Finally, for all $t\geqslant s\geqslant 0$,
\[\sum_{i=1}^N  \left|\lambda\po \bar U_i^N(t)\pf  -  \lambda\po \bar U_i^N(s)\pf  \right| \ \leqslant \ N \alpha\lambda_* (t-s) + N kh \lambda_* (t-s) + \lambda_*\sum_{i=1}^N\int_s^t \int_{\R_+}\indiq_{\{z\leqslant \lambda_*\}}\pi^i(dz,du)\,.\]
In particular, taking the expectation,
\[|z_t-z_s| \ \leqslant \ (t-s)\lambda_*\po \alpha + kh+\lambda_*\pf\,.\]
By Doob's maximal inequality
\[\mathbb P \po \sup_{v\in[s,t]}\co \sum_{i=1}^N\int_s^v \int_{\R_+}\indiq_{\{z\leqslant \lambda_*\}}\pi^i(dz,du) - N(v-s)\lambda_*\cf \geqslant N \sqrt{t-s}\pf \ \leqslant \ \frac{ \sqrt{\lambda_*}}{\sqrt N}\,. \]
From now on, fix $\varepsilon>0$ and $t >  0$. Chose a subdivision $0=t_1\leqslant \dots \leqslant t_K = t$ so that 
\[2 \lambda_*  (t_{i+1}-t_i)(\alpha+kh+\lambda_*) + \lambda_*  \sqrt{t_{i+1}-t_i} \ \leqslant \ \frac\varepsilon3\,, \qquad i\in\cco 1,K-1\ccf\,.\]

This ensures that
\begin{multline*}
H^1 \ :=\ \left\{\sup_{i\in\cco 1,K-1\ccf}\sup_{u\in[t_i,t_{i+1}]} \po |z_u - z_{t_i}| + \frac1N\sum_{j=1}^N  \left|\lambda\po \bar U_j^N(u)\pf  -  \lambda\po \bar U_j^N(t_i)\pf  \right| \pf \geqslant  \frac{\varepsilon}{3} \right\}\\
\subset \ \bigcup_{i=1}^{K-1} \left\{\sup_{v\in[t_i,t_{i+1}]}\co \sum_{j=1}^N\int_{t_i}^v \int_{\R_+}\indiq_{\{z\leqslant \lambda_*\}}\pi^j(dz,du) - N(v-t_i)\lambda_*\cf \geqslant N \sqrt{t_{i+1}-t_i} \right\} ,
\end{multline*}
and thus
\[\mathbb P(H^1) \ \leqslant \ \frac{K  \sqrt{\lambda_*}}{\sqrt N}\,.\]
Moreover, for all $i\in\cco 1,K\ccf$,
\[\mathbb P \po \left|z_{t_i} - \frac1N \sum_{j=1}^N \lambda\po \bar U_j^N({t_i})\pf \right| \geqslant \frac\varepsilon 3\pf  \ \leqslant \ \frac{3\lambda_*}{\varepsilon \sqrt N}\,.\]
We can finally bound
\begin{multline*}
\mathbb P \po \sup_{s\in[0,t]} \left| z_s - \frac1N \sum_{i=1}^N \lambda \po U_i^N(s)\pf\right| \ \geqslant \ \varepsilon \pf   \  \leqslant \ \mathbb P \po \sup_{s\in[0,t]} \sum_{i=1}^N |U_i^N(s)-\bar U_i^N(s)|  \geqslant \frac{\varepsilon}{ 3 k} N\pf \\
+ \mathbb P(H^1)  + \sum_{i=1}^K \mathbb P \po \left|z_{t_i} - \frac1N \sum_{j=1}^N \lambda\po \bar U_j^N(t_i)\pf \right| \geqslant \frac\varepsilon 3\pf 
\end{multline*}
and conclude with the bounds already proven.
\end{proof}

From Propositions \ref{prop:Z}  and \ref{prop:propchaos} we can get a more quantitative version of Proposition~\ref{prop:delta0instable1}. Notice that $ r = \alpha $ under Assumption \ref{ass:lambda}.

\begin{prop}\label{PropOinstable}
Grant Assumption  \ref{ass:lambda} and assume moreover that $kh>\lambda_*+\alpha$.  Let $\bar U$ be a solution of \ref{eq:limitU}, and let $(x_t)_{t\geqslant0}$ be the solution of the auxiliary limit equation \eqref{eq:limitz} with $x_0 = \E\po \lambda(\bar U(0))\pf \wedge z_\infty$. Then, for all $t\geq 0$,
\begin{equation}\label{eq:firstresult}
\E\po \lambda(\bar U(t))\pf \ \ge x_t\,.
\end{equation}
In particular, 
\[\E\po \lambda(\bar U(0))\pf >0 \qquad \Rightarrow \qquad \liminf_{t\rightarrow\infty} \E\po \lambda(\bar U(t))\pf \ \ge \ x_\infty \ > \ 0\, , \]
and if $g$ is an equilibrium distribution for $\bar U$ and $p_* = \int_0^\infty \lambda(x)g(dx)$, then $p_* \geqslant x_\infty$.
\end{prop}
\begin{proof}
Consider for all $N\in\mathbb N$ the coupling $(U^N,\bar U^N)$ introduced in Proposition \ref{prop:propchaos} with initial conditions independent and with the same law as $\bar U(0)$. Then
\[ \left| \E\po \lambda(\bar U(t))\pf - \frac1N\sum_{i=1}^N\E\po \lambda( U_i^N(t))\pf\right| \ \leq \ \frac{k}N\sum_{i=1}^N \E\po |\bar U_i^N(t) - U_i^N(t)|\pf \ \underset{N\to \infty}{\longrightarrow}\  0\,.\]
Proposition \ref{prop:Z} and Theorem \ref{theo:ldp} then yields
\[\E\po \lambda( U_1^N(t))\pf  \ = \ \frac1N \sum_{i=1}^N \E\po \lambda( U_i^N(t))\pf  \ \geq \ \E\po Z^N (t) \pf \ \ \underset{N\to \infty}{\longrightarrow}\ x_t\,,\]
which proves the first two assertions of the proposition. Letting $t$ go to infinity with $\bar U_0$ distributed according to $g$ gives the third one.

\end{proof}

\subsection{Longtime convergence for the limit process} \label{sec:couplage-non-lineaire}
This section is devoted to the proof of Theorem~\ref{theo:gattractive}. The condition \eqref{condition_a_b_y0} is enforced. In this section  we consider two versions $ \bar U $ and $ \tilde U$ of the limit process with different initial distributions $\eta_0$ and $\tilde \eta_0$, and we write   $Z=\lambda(\bar U)$, $\tilde Z=\lambda(\tilde U)$,  $z_t =\E( Z(t))$ and $ \tilde z_t = \E (\tilde Z(t))$.

\begin{theo}\label{theo:CouplNonLin}
Under Assumption \ref{ass:lambda} and provided \eqref{condition_a_b_y0} holds,  there exist explicit constants $C,\delta>0$ and $\kappa\in(0,1)$ with the following property. If $z_0\wedge \tilde z_0 \geqslant x_\infty-\delta, $ then, considering the synchronous coupling of $U$ and $\tilde U$, for all $t\geq 0$,
\[ \E\po | Z(t)-\tilde Z(t)|\pf\ \ \le \ C \kappa^t \E\po | Z(0)-\tilde Z(0)|\pf\,.\]
\end{theo}

\begin{proof}
The key argument is that, under the condition \eqref{condition_a_b_y0}, due to their positive drifts, the processes spend a sufficiently small time in $[0,\lambda_*/k]$, which is the only region where they can have different jump rates, and thus where there can be asynchronous jumps, or where $|z_t-\tilde z_t|$ has an influence on the evolution of the jump rates $Z$ and $\tilde Z$.

{\bf Preliminaries.}
First, from Proposition~\ref{PropOinstable}, for any $\delta>0,$ if $z_0\wedge \tilde z_0 \geq x_\infty-\delta ,$  then $z_t\wedge \tilde z_t \geq x_\infty-\delta$ for all $t\geq 0$. As a consequence, the drift felt by $\bar U$ at time $s$ and position $x\in[0,\lambda_*/k]$  is  
$$ -\alpha x +   h   z_s \  \geq \  - \frac1k \alpha \lambda_* + h(x_\infty-\delta) \ = \ \lambda_*h(1-2a-b) -h\delta\,.  $$
From now on we chose $\delta>0$ small enough so that the right hand side is positive.   This means that, in the absence of jumps, $\bar U$ is non-decreasing as long as it is in $[0,\lambda_*/k]$, and that the crossing of level $ \lambda_*/k $ is an up-crossing  for $ \bar U$. Once it has crossed this level, it cannot come back to levels strictly below $ \lambda_*/k$ (without jumping). More precisely, denoting $T_1 = \inf \{ t > 0 : 0 = \bar U (t ) < \bar U(t-) \} $ the first jump time of $ \bar U$,  for $t<T_1$,
$$ \bar U(t)  \geq \int_0^t e^{- \alpha (t-s) } h z_s   ds  \geq  h (x_\infty-\delta)\int_0^t  e^{- \alpha (t- s)} ds = \frac{  h (x_\infty-\delta)}{ \alpha} (1 - e^{- \alpha t }),$$
which is larger than $\lambda_*/k$ for all $t\ge t_\delta$ with
\begin{equation}\label{eq:to}
t_\delta \ := \ -\frac{1}{\alpha} \ln \Big( 1 - \frac{  \alpha \lambda_* }{  kh(x_\infty-\delta) } \Big) \ \underset{\delta\rightarrow 0}\longrightarrow \ -\frac{1}{\alpha} \ln \Big( 1 - \frac{  a }{ 1-a-b} \Big) \ = \  \frac{1}{\alpha} \ln \Big( 1 + \frac{  a }{ 1-2a-b} \Big) \,.
\end{equation}
 The same holds for $\tilde U$.   Denoting $ \tilde T_1$  the first jump time of $ \tilde U$, we have for all $  t_\delta \le t < T_1 \wedge \tilde T_1  , $ 
$$ \bar U(t) \vee  \tilde U(t)  \geq \lambda_*/k \,,\qquad \mbox{ whence }\qquad Z  (t)   = \lambda_* =  \tilde Z(t)\,.$$

\bigskip

\noindent
{\bf Step 1.} Let $  t > 0  $ and consider the last jump time before time $t,$ 
$$ L_t = \max \{ 0 \le s \le t : \tilde U (s-) \neq \tilde U(s) \mbox{ or } \bar U(s-) \neq \bar U(s) \} , \qquad \text{ with } \max \emptyset := 0\,.$$ 
Remark that if $ L_t \le t- t_\delta$, then $Z(t)=\tilde Z(t)=\lambda_*$. 

{\it Case 1.} If $L_t>(t-t_\delta)_+$ is a synchronous jump of $ \bar U $ and $ \tilde U$, 
then both $ \bar U$ and $ \tilde U$ are reset to $0$ at time $L_t,$ such that on this event, 
\begin{equation}\label{eq:zt}
|Z(t)-\tilde Z(t)| \ \le\  k| \bar U(t) - \tilde U(t) |
 \le \  kh \int_{(t-t_\delta)_+}^t e^{- \alpha (t-s) }  | z_s - \tilde z_s | ds 
 .
\end{equation} 
Besides,
\[\mathbb P\po \text{synchronous jump in }[(t-t_\delta)_+,t] \pf \ \le \ 1- e^{-\lambda_* t_\delta} \ \le \ \lambda_* t_\delta\,,\]
whence
\[\E \po | Z_t - \tilde Z_t | \indiq_{\{ L_t \mbox{\tiny \ synchronous jump}\}}\pf \ \le \ kh \lambda_* t_\delta \int_{(t-t_\delta)_+}^t e^{- \alpha (t-s) }  | z_s - \tilde z_s | ds .\]

{\it Case 2.}  If $ L_t > (t-t_\delta)_+$ is an asynchronous jump, i.e. a jump only for one of the two processes, then, since $ | Z(t) - \tilde Z(t) | \le \lambda_*,$ we may simply upper bound 
$$ | Z (t) - \tilde Z(t) | \indiq_{\{ L_t \mbox{\tiny \ asynchronous jump}\}} \ \le\  \lambda_* \indiq_{\{ L_t \mbox{\tiny \ asynchronous jump}\}} \indiq_{\{ L_t \geq (t -t_\delta)_+ \}}.$$
The probability of having an uncommon jump between $ t- t_\delta $ and $t$ is upper bounded by the integral of the expectations of the differences of the intensities, that is, by
$$\P ( \mbox{asynchronous jump in } [(t -t_\delta)_+, t] ) \le  \int_{(t-t_\delta)_+}^t \E (| Z(s) - \tilde Z(s) ) | ds  . $$ 
Therefore,
$$ \E \Big(  | Z (t) - \tilde Z(t) | 1_{\{ L_t \mbox{\tiny \ asynchronous jump}\}} \Big) \le  \lambda_*  \int_{(t-t_\delta)+}^t \E (| Z(s) - \tilde Z(s) ) | ds  .$$ 

{\it Case 3.} Finally, suppose that no jump has happened during $ [0, t ].$ This case is only interesting for $ t < t_\delta, $ because otherwise $ Z(t)= \tilde Z(t) = \lambda_*.$ The same goes if $\bar U(0)\wedge \tilde U(0)\geq \lambda_*/k$. If $\bar U(0)\vee \tilde U(0)\leq \lambda_*/k$,
\begin{eqnarray*}
|Z(t)-\tilde Z(t)| \ \le\  k| \bar U(t) - \tilde U(t) | \ \le \  kh \int_{0}^t e^{- \alpha (t-s) }  | z_s - \tilde z_s | ds + e^{- \alpha t } | Z(0 ) - \tilde Z (0) |\,,
\end{eqnarray*}
where we used that $k \bar U(0)=Z(0)$ and $k\tilde U(0) = \tilde Z(0)$. If one of the processes (say $\bar U$) starts below $\lambda_*/k$ and the other above, we are brought back to the previous case by considering the solution $\hat U$ of $d\hat U= - \alpha \hat u + z_s$ with $\hat U(0)=\lambda_*/k$, in which case $Z(s)=\lambda_* = \lambda(\hat U(s))$ for all $s\in[0,t]$ and $Z(0)=\hat U(0)$. So, in all cases, we have obtained that 
\[|Z(t)-\tilde Z(t)|\indiq_{\{\mbox{\tiny no jump in }[0,t]\}} \ \le \ kh \int_{0}^t e^{- \alpha (t-s) }  | z_s - \tilde z_s | ds + e^{- \alpha t } | Z(0 ) - \tilde Z (0) |\,.\]  

{\it Conclusion of Step 1.} Putting these three cases together and writing for short $f(t) = \E | Z (t) - \tilde Z(t)|$, we conclude that for all $t\geq 0$, 
\begin{eqnarray}\label{eq:gronwall}
 f(t) & \le & \nu_\delta \int_{(t-t_\delta)_+}^t  f(s) ds    
+  \left[ kh  \int_0^t  f(s) ds  + f(0) \right] \indiq_{ [0, t_\delta]} (t) \,,
\end{eqnarray} 
with $\nu_\delta  =kh\lambda_* t_\delta+\lambda_*$.

{\bf Step 2.} Using Gronwall's inequality, for all $t \le t_\delta $,
\[f(t)   \ \leq \ f(0) e^{\po \nu_\delta +kh\pf t }\,.\]
As a consequence, for $ t \in  [ t_\delta, 2 t_\delta ],$ 
\[ f(t) \ \le \ \nu_\delta \left[ \int_{t-t_\delta} ^{t_\delta}  f(s) ds +    \int_{t_\delta}^t  f(s) ds \right] 
\ \le \ \nu_\delta \left[ \int_{t_\delta}^{t}  f(s) ds +    f(0) e^{\po \nu_\delta +kh\pf t_\delta } t_\delta \right] \,, \]
and Gronwall's inequality implies 
$$ f(t) \le \nu_\delta t_\delta e^{(2\nu_\delta+kh) t_\delta}   f(0)  $$
for all $ t \in  [ t_\delta, 2 t_\delta ].$  Iterating the above inequality over time intervals $ [n t_\delta, (n+1) t_\delta[ $ for $n\geq2$, we obtain that for all $ t \in [n t_\delta, (n+1) t_\delta[, $
$$ f(t)  \le (\nu_\delta t_\delta e^{\nu_\delta t_\delta})^n e^{ (\nu_\delta+kh)  t_\delta}  f(0) \,.
$$ 
Since  $\ln(1+ x) < x$ for all $x> 0$,
\begin{eqnarray*}
\nu_\delta t_\delta \ \underset{\delta\rightarrow 0}\longrightarrow \  \lambda_* t_0 \po  1 + kht_0\pf &= & \frac{b}{a}\ln \Big( 1 + \frac{  a }{ 1-2a-b} \Big)  \po 1+ \frac1a\ln \Big( 1 + \frac{  a }{ 1-2a-b} \Big) \pf \\
&   < & \frac{b}{1-2a-b}\po 1 + \frac{1}{1-2a-b}\pf \ \leqslant \ y_0\,,
\end{eqnarray*}
(this is the condition \eqref{condition_a_b_y0}). As a consequence, we can chose  $\delta>0$ small enough so that $\nu_\delta t_\delta e^{\nu_\delta t_\delta} <1$. Using that $ n+1 \geq t/t_\delta$ for $t\in [n t_\delta,(n+1)t_\delta]$, we have thus obtained
\[f(t) \ \leq\ \po \co \nu_\delta t_\delta  e^{ \nu_\delta t_\delta}\cf^{1/t_\delta}\pf^t \frac{e^{kht_\delta}}{\nu_\delta t_\delta}  f(0) \,,\]
which concludes the proof.
\end{proof}

\begin{cor}\label{Cor:CouplNonLin}
Under Assumption \ref{ass:lambda} and provided \eqref{condition_a_b_y0} holds,  let $\kappa$ be given by Theorem~\ref{theo:CouplNonLin}. For all $\gamma>0$, there exists $C_\gamma>0$ such that for all initial conditions with $z_0\wedge \tilde z_0 \geqslant \gamma$, considering the synchronous coupling of $U$ and $\tilde U$, for all $t\geq 0$,
\[ \E\po | Z(t)-\tilde Z(t)|\pf\ \ \le \ C_\gamma \kappa^t \E\po | Z(0)-\tilde Z(0)|\pf\,.\]
\end{cor} 

\begin{proof}
Let $C,\delta>0$ be given with $\kappa$ in Theorem~\ref{theo:CouplNonLin}.  By Proposition~\ref{PropOinstable}, if $z_0\wedge \tilde z_0 \geqslant \gamma$ then $z_t\wedge \tilde z_t \geqslant x_\infty - \delta$ for all $t$ larger than some finite $T_\gamma$, which is the time taken by the solution of the  auxiliary limit equation \eqref{eq:limitz} to reach $x_\infty - \delta$. Hence, after time $T_\gamma$, we can apply Theorem~\ref{theo:CouplNonLin} to get that
\[ \E\po | Z(t)-\tilde Z(t)|\pf\ \ \leqslant \ C \kappa^{t-T_\gamma} \E\po | Z(T_\gamma)-\tilde Z(T_\gamma)|\pf\]
for all $t\geqslant T_\gamma$. It only remains to control the small times $t\leqslant T_\gamma$. The proof is similar to the proof of Theorem~\ref{theo:CouplNonLin}, except that now we don't need to prove a contraction, only a rough bound.

For $t\geqslant 0$, let $L_t$ be the time of the last jump before time $t$. If $L_t=0$ (no jump) or if $L_t$ is a synchronous jump of $\bar U$ and $\tilde U$, we simply bound
\[|Z(t)-\tilde Z(t)| \ \leqslant \ |Z(L_t)-\tilde Z(L_t)| + kh \int_{L_t}^t|z_s-\tilde z_s|ds\,.\]
If $L_t$ is an asynchronous jump, as in the proof of Theorem~\ref{theo:CouplNonLin} we simply bound $|Z(t)-\tilde Z(t)| \leqslant \lambda_*$ and then bound the probability to have an asynchronous jump in $[0,t]$ by $\int_0^t \mathbb E(|Z(s)-\tilde Z(s)|ds$. Gathering these two cases we get that for all $t\geqslant 0$
\[ \E\po | Z(t)-\tilde Z(t)|\pf\ \ \leqslant \   \E\po | Z(0)-\tilde Z(0)|\pf + (kh+\lambda_*) \int_0^t   \E\po | Z(s)-\tilde Z(s)|\pf ds \,.\]
By Gronwall's Lemma,
\[ \E\po | Z(t)-\tilde Z(t)|\pf\ \ \leqslant \  e^{(kh+\lambda_*)  t} \E\po | Z(0)-\tilde Z(0)|\pf   \]
for all $t\geqslant 0$, which concludes.
\end{proof}
 
A straightforward corollary of this result is then 
 
\begin{cor}\label{Cor:uniqueEquilibre}
Under Assumption \ref{ass:lambda} and provided  \eqref{condition_a_b_y0} holds, the non-linear system \eqref{eq:limitU} admits exactly one non-zero equilibrium $g$ as described in Theorem \ref{theo:41}. Moreover, for all solutions of the non-linear system with $z_0 >0$, we have that $z_t\to  \int_0^\infty \lambda(x)g(x)dx$ as $t\to\infty$.
\end{cor}

The contraction at the level of the jump rates then yields a Wasserstein contraction at the level of the processes~:

\begin{proof}[Proof of Theorem~\ref{theo:gattractive}]
Let $\kappa,C,\delta$ be given by Theorem~\ref{theo:CouplNonLin} and $t_\delta$ as defined in the proof of the latter. Let $L_t$ be the last jump before time $t$ ($L_t = 0$ if there is no jump). Then
 \begin{eqnarray}\label{eq:UtildeU}
|\bar U(t) - \tilde U(t) | & \le & h \int_{L_t}^t e^{-\alpha (t-s)}|z_s-\tilde z_s|ds + e^{-\alpha (t-L_t)}|\bar U(L_t)-\tilde U(L_t)| \,. 
 \end{eqnarray}
%

\textbf{Step 1.}  Consider first the case where $z_0\wedge \tilde z_0 \geq x_\infty - \delta$. 
 First, if $L_t$ is a synchronous jump of $\bar U$ and $\tilde U$, the last term of \eqref{eq:UtildeU} is zero. Second, if it is an asynchronous jump, say $\bar U$ has jumped but not $\tilde U$, then necessarily $\lambda(\tilde U(L_t)) < \lambda_*$ so that $\tilde U(L_t)\leq \lambda_* /k $, while $\bar U(L_t)=0$. As in Step 1--Case 2 of the proof of Theorem \ref{theo:CouplNonLin}, we bound
\begin{eqnarray*}
\E \po  | \bar U(L_t) - \tilde U(L_t) | 1_{\{ L_t \mbox{\tiny \ asynchronous jump}\}} \pf  & \le  & \frac{ \lambda_*}{k}  \int_{(t-t_\delta)_+}^t \E (| Z(s) - \tilde Z(s) ) | ds  \\
&\leq & \lambda_*  t_\delta C \kappa^{t-t_\delta} \E \po  | \bar U(0) - \tilde U(0) |\pf\,  .
\end{eqnarray*}
Third, we bound
 \begin{eqnarray*}
 \int_{L_t}^t e^{-\alpha (t-s)}|z_s-\tilde z_s|ds  & \leqslant & \indiq_{\{L_t>t/2\}}\int_{t/2}^t  |z_s-\tilde z_s|ds +\indiq_{\{L_t\le t/2\}} \int_{0}^t |z_s-\tilde z_s|ds \\
 & \leqslant & Ck \po \frac {t\kappa^{t/2}}2  + \frac{1}{\ln(1/\kappa)}\indiq_{\{L_t\le t/2\}}\pf \E \po  | \bar U(0) - \tilde U(0) |\pf\,.
 \end{eqnarray*}
 The event $\{L_t\le t/2\}$ implies that there is no jump on the time interval $[t/2,t]$. In that case, as in the proof of Theorem~\ref{theo:CouplNonLin}, $\bar U$ and $\tilde U$ will be above the level $\lambda_*/k$ for all times larger than $t/2+t_\delta$, in which case their jump rates will be $\lambda_*$. As a consequence,
 \[\mathbb P\po L_t \leqslant t/2\pf \ \leqslant \ e^{-\lambda_* (t/2-t_\delta)}\,.\]
Gathering all the previous bounds, we have finally obtained that, in the case where $z_0\wedge \tilde z_0 \geq x_\infty - \delta$, for all $t\geqslant 0, $
\begin{eqnarray*}
\E\po |\bar U(t) -\tilde U(t)|\pf & \leqslant & \co e^{-\alpha t} +  \lambda_*  t_\delta C \kappa^{t-t_\delta}+ Ckh \po \frac {t\kappa^{t/2}}2  + \frac{e^{-\lambda_* (t/2-t_\delta)}}{\ln(1/\kappa)}\pf \cf \E \po  | \bar U(0) - \tilde U(0) |\pf\\
& \leqslant & \tilde C \tilde \kappa^t\E \po  | \bar U(0) - \tilde U(0) |\pf
\end{eqnarray*}
for some $\tilde C\geqslant 1$, $\tilde \kappa \in (0,1)$.
 
 \textbf{Step 2.} Now we only suppose that $z_0\wedge\tilde z_0 \geqslant \gamma$ for some $\gamma>0$. Since $z_t\wedge \tilde z_t$ will reach $x_\infty-\delta$ in a finite time $T_\gamma$, as in the proof of Corollary~\ref{Cor:CouplNonLin}, it only remains to obtain a rough bound for small times $t\leqslant T_\gamma$. Using that
 \[\E \po  | \bar U(L_t) - \tilde U(L_t) | 1_{\{ L_t \mbox{\tiny  \ asynchronous jump}\}} \pf  \ \le  \  \frac{ \lambda_*}k  \mathbb P\po \text{there is an asynchronous jump in }[0,t]\pf\,,  \] 
 we get from \eqref{eq:UtildeU} that 
\begin{eqnarray*}
\mathbb E \po |\bar U(t) - \tilde U(t) |\pf   & \leqslant &  h \int_{0}^t |z_s-\tilde z_s|ds + \mathbb E\po |\bar U(0)-\tilde U(0)|\pf + \frac{\lambda_*}{k}\int_0^t \mathbb E\po |Z(s)-\tilde Z(s)|\pf ds \\
& \leqslant & (kh+\lambda_*) \int_0^t \mathbb E\po |\bar U(s)-\tilde U(s)|\pf ds + \mathbb E\po |\bar U(0)-\tilde U(0)|\pf\,,
\end{eqnarray*}
and thus, for all $t\geqslant 0$,
\[\mathbb E \po |\bar U(t) - \tilde U(t) |\pf  \ \leqslant \ e^{(kh+\lambda_*)t} \mathbb E\po |\bar U(0)-\tilde U(0)|\pf\,.\] 
Using the result of the first step, we conclude the proof of the first claim~: for all $t\geqslant 0$,
\begin{eqnarray*}
\mathbb E \po |\bar U(t) - \tilde U(t) |\pf  & \leqslant & \tilde C \tilde \kappa^{(t-T_\gamma)_+}  \mathbb E \po |\bar U(T_\gamma\wedge t) - \tilde U(T_\gamma\wedge t) |\pf\\
& \leqslant & \tilde C \tilde \kappa^{t-T_\gamma}  e^{(kh+\lambda_*)T_\gamma} \mathbb E\po |\bar U(0)-\tilde U(0)|\pf\,.
\end{eqnarray*}

The two other claims are immediately obtained by choosing $(U(0),\tilde U(0))$ according to a $W_1$-optimal coupling of $\eta_0$ and $\tilde \eta_0$, and by considering the case where $\tilde \eta_0=g$.
\end{proof}

\section{Exponentiality of exit times: a general result}\label{sec:generalresult}

In this section we give general conditions that ensure that the rescaled exit times of a  Markov process are close in law to the exponential law, extending the results of \cite{brassesco} for low-noise diffusion processes. Let $(X_t)_{t\geqslant 0}$ be a time-homogeneous strong Markov process taking values in some Polish space $E$. Let $\mathcal D,\mathcal K$ be   measurable subsets of $E$ with $\emptyset \neq \mathcal K \subset \mathcal D$. For $A\subset E$, denote
\[\tau_A \ = \ \inf\{t\geqslant 0\ : \ X_t \in A\}\,.\]
Let $\varepsilon_1,\varepsilon_2,\varepsilon_3 \in [0,1]$ and $s_1\geqslant s_2>0$ be such that
\begin{eqnarray}
\varepsilon_1 & \geqslant & \sup_{x\in\mathcal K} \ \mathbb P_x \po \tau_{\mathcal D^c} \leqslant s_1 \pf \label{eqdef:generalexit3}  \\
 \varepsilon_2 & \geqslant & \sup_{x\in\mathcal D} \ \mathbb P_x \po \tau_{\mathcal D^c} \wedge   \tau_{\mathcal  K} > s_2 \pf \label{eqdef:generalexit11} \\
\varepsilon_3  & \geqslant  &   \sup_{t \geq 0} \sup_{x,y\in\mathcal K}\left|\mathbb P_x \po \tau_{\mathcal D^c} > t  \pf - \mathbb P_y \po \tau_{\mathcal D^c} >t  \pf\right| \label{eqdef:generalexit1}\,.
\end{eqnarray}
At the end of this section (Proposition~\ref{prop:eps1}) we provide a general coupling argument to bound $\varepsilon_3$.

To fix ideas, the set $\mathcal K$ can be thought of as a metastable set, far from the boundary of $\mathcal D$ but in which the process spends most of its time as long as it hasn't left $\mathcal D$ (for a diffusion process with small noise, $\mathcal K$ would be a neighborhood of a fixed point of the deterministic ODE at zero temperature). Having in mind that $\varepsilon_1,\varepsilon_2$ and $\varepsilon_3$ are meant to be small, the conditions \eqref{eqdef:generalexit3}, \eqref{eqdef:generalexit11} and \eqref{eqdef:generalexit1} typically  mean that,  whatever its initial condition in $\mathcal D$, if the process hasn't already left $\mathcal D$ in a time $s_2$ then it has probably reached $\mathcal K$, and then it will typically stay in $\mathcal D$ at least for a time $s_1$ and forget its initial position in $\mathcal K$.

In order to state the main result of this section, we fix some $x_0 \in \mathcal K$. We would like to consider  $\beta>0$ such that $\mathbb P_{x_0} \po \tau_{\mathcal D^c} > \beta\pf \in [1/4,3/4]$. While the existence of such $\beta$ is clear in the case of diffusion processes, for which $t\mapsto \mathbb P_x\po \tau_{\mathcal D^c} > t\pf $ is continuous, it is not necessarily easy  to check in general. Nevertheless we can state the following.

\begin{ass}\label{hyp:exit}
We have $\varepsilon_1+\varepsilon_2 +\varepsilon_3 \leqslant 1/2$ and $\tau_{\mathcal D^c}$ is $\mathbb P_{x_0}$-almost surely finite. 
\end{ass}

\begin{lem}\label{lem:beta-exitgene}
Under Assumption~\ref{hyp:exit}, there exists $\beta>0$ such that $\mathbb P_{x_0} \po \tau_{\mathcal D^c} > \beta\pf \in [1/4,3/4]$.
\end{lem}
We postpone the proof.

Under Assumption~\ref{hyp:exit}, it will be convenient to consider  $C,\delta,M>0$ such that
\begin{equation}\label{eq:cond-general}
\frac{s_2}\beta \vee (\varepsilon_2 + \varepsilon_3) \ \leqslant \ C e^{-\delta M}\,,
\end{equation}
where $\beta$ is such that $\mathbb P_{x_0} \po \tau_{\mathcal D^c} > \beta\pf \in [1/4,3/4]$. Obviously, it is always possible to find some constants such that \eqref{eq:cond-general} holds but, for fixed $C,\delta$, most of the results in this section are interesting only for $M$ large enough. 

For a system of $N$ interacting particles as studied in the paper, $M$ would typically be $N$ (but it could also be $\sqrt N$, $\ln N$, etc.), and for a diffusion process at small temperature, $M$ would be the inverse temperature.

The aim of this section is to establish the following result.

\begin{theo}\label{thm:gene-exit}
For all $\delta,C>0$, there exists $M_0$ such that, if Assumption~\ref{hyp:exit} and \eqref{eq:cond-general} hold for some $M\geqslant M_0$, then we have the following properties. First, 
\[\sup_{y\in\mathcal D} \mathbb E_{y}(\tau_{\mathcal D^c}) \  <\  +\infty\,.\]
Second, for all $K,\theta>0$, there exist $K'>0$ (that depends on $\delta,C,K,\theta$ but not $M$) such that for all probability measures $\nu_1$ and $\nu_2$ on $\mathcal D$ such that
\[ \mathbb P_{\nu_i}\po \tau_{\mathcal D^c} \leqslant  s_2\pf \ \leqslant K e^{-\theta M}\,,\qquad i=1,2\,,\]
we have
\begin{eqnarray}\label{eq:thm-gene-exit-1}
\sup_{t\geqslant 0} |\mathbb P_{\nu_1}\po \tau_{\mathcal D^c} > t \mathbb E_{\nu_1}\po \tau_{\mathcal D^c} \pf \pf - e^{-t}|  & \leqslant & K' M^3e^{-\min (\delta /3  ,1/2,\theta) M}
\end{eqnarray}
and 
\begin{eqnarray}\label{eq:thm-gene-exit-2}
\left| \frac{\mathbb E_{\nu_1}\po \tau_{\mathcal D^c} \pf}{\mathbb E_{\nu_2}\po \tau_{\mathcal D^c} \pf}-1\right| & \leqslant & K' M^3e^{-\min (\delta/3,1/2,\theta) M}\,.
\end{eqnarray}
\end{theo}


The remainder of this section is devoted to the proof of Theorem \ref{thm:gene-exit}. The general strategy is the following: for $t\geqslant0$, denote $f_0(t) = \mathbb P_{x_0}\po \tau_{\mathcal D^c} > t\pf$. Due to the Markov property, the loss of memory in $\mathcal K$ provided by \eqref{eqdef:generalexit1}, and the fast return to $\mathcal K$ after each excursion out of $\mathcal K$ provided by \eqref{eqdef:generalexit11}, we get that the function $f_0$ approximately satisfies the relation $f(t+s)=f(t)f(s)$, for $s,t\geqslant 0$, which characterizes functions of the form $t\mapsto e^{-ct}$ for some $c\in \R$. Obtaining, from this approximate functional relation, that $f_0$ is necessarily close to an exponential function, is then slightly technical but elementary.

More precisely, the main point is the following result (from which, in particular, Lemma~\ref{lem:beta-exitgene} will easily follow).
\begin{lem}\label{lem:exit_gen1}
For all  $t\geqslant 0$ and for all $s\geqslant s_2$,
\[f_0(t)  \po f_0\po s \pf - \varepsilon_2 - \varepsilon_3\pf \ \leqslant \ f_0(t+s) \ \leqslant \ f_0(t) \po f_0\po s - s_2\pf + \varepsilon_2 + \varepsilon_3\pf\,.\]
\end{lem}
\begin{proof}
 For $t\geqslant 0$, consider the event
\[A_t \ = \ \{\forall s\in [t ,t +s_2],\ X_s \notin \mathcal K\}\,.\]
By the Markov property, conditioning on the value of $X_{t }$, for all $t>0$,
\begin{eqnarray*}
\mathbb P_{x_0}\po A_t\,,\ \tau_{\mathcal D^c}  > t   + s_2\ |\ \tau_{\mathcal D^c} > t   \pf & \leqslant & \varepsilon_2\,,
\end{eqnarray*}
and thus for all $s\geqslant   s_2 $,
\begin{eqnarray*}
\mathbb P_{x_0}\po A_t\,,\ \tau_{\mathcal D^c} > t+s   \pf & \leqslant & \varepsilon_2 f_0(t)\,.
\end{eqnarray*}
Similarly, by the strong Markov property, for all $s\geqslant  s_2 $, by conditioning on $X_H$ where $H=\inf\{u>t ,\ X_u\in\mathcal K\}$,
\begin{eqnarray*}
\mathbb P_{x_0}\po \tau_{\mathcal D^c} > s+t \,, \ \overline{A_t}\pf & \leqslant & f_0(t) \sup_{y\in\mathcal K} \mathbb P_{y}\po  \tau_{\mathcal D^c} > s   -s_2 \pf \\
& \leqslant & f_0(t) \po f_0\po s -  s_2 \pf + \varepsilon_3\pf
\end{eqnarray*}
 and
\begin{eqnarray*}
\mathbb P_{x_0}\po \tau_{\mathcal D^c} >  s+t \,, \ \overline{A_t}\pf  & \geqslant & f_0(t) \mathbb P_{x_0} \po \tau_{\mathcal D^c}  > t  + s_2,\  \overline{A_t} |\tau_{\mathcal D^c}  > t \pf \inf_{y\in\mathcal K} \mathbb P_{y}\po  \tau_{\mathcal D^c} > s     \pf \\
& \geqslant & f_0(t) \po 1 - \varepsilon_2 \pf\po f_0\po s \pf - \varepsilon_3\pf\\
& \geqslant & f_0(t)  \po f_0\po s \pf - \varepsilon_2 - \varepsilon_3\pf\,.
\end{eqnarray*}
\end{proof}

\begin{proof}[Proof of  Lemma~\ref{lem:beta-exitgene}]
First, $f_0(0)=1$ and, since  $\tau_{\mathcal D^c}$ is $\mathbb P_{x_0}$-almost surely finite, necessarily $f_0(t)$ vanishes as $t\rightarrow +\infty$. Applying Lemma~\ref{lem:exit_gen1} with $s=s_2$ and using that $f_0$ is a non-increasing function we get that for all $t\geqslant 0$
\[f_0(t) \ \geqslant \ f_0(t+s_2) \ \geqslant \ f_0(t) \po f_0(s_2) - \varepsilon_2-\varepsilon_3\pf \ \geqslant \ f_0(t) \po 1 - \varepsilon_1-\varepsilon_2-\varepsilon_3\pf \]
where for the last inequality we used \eqref{eqdef:generalexit3}  and the fact $s_2\leqslant s_1$. Thus, if $\varepsilon_1+\varepsilon_2+\varepsilon_3\leqslant 1/2$, we get that for all $t\geqslant 0$,
\[|f_0(t+s_2) - f_0(t)| \ \leqslant \ \frac12 f_0(t) \ \leqslant \ \frac12\,.\]
In particular any interval of $[0,1]$ with length at least $1/2$ intersects $\{f_0(t),\ t\geqslant 0\}$, which concludes.
\end{proof}

From now on we work under Assumption~\ref{hyp:exit}, we take $\beta $ as in Lemma \ref{lem:beta-exitgene}, consider $C,\delta,M>0$ such that \eqref{eq:cond-general} holds, and write $\varepsilon = \varepsilon_2+\varepsilon_3$ and  $s_0 = -\ln \mathbb P_{x_0}(\tau_{\mathcal D^c}>\beta)$. For $t\geqslant 0$, let
\[f(t) \ = \ \mathbb P_{x_0} \po \tau_{\mathcal D^c}  >  t\beta/s_0\pf\,.\]
By construction, $f$ is a non-increasing function with $f(0)=1$ and $f(s_0) = e^{-s_0}\in[1/4,3/4]$. Moreover,  Lemma~\ref{lem:exit_gen1} states that for all $t\geqslant 0$ and for all $s\geqslant s_2 s_0/\beta$,
\begin{equation}\label{eq:ineq-general}
f(t)  \po f\po s \pf - \varepsilon\pf \ \leqslant \ f(t+s) \ \leqslant \ f(t) \po f\po s - s_2 s_0/\beta\pf + \varepsilon\pf\,.
\end{equation}

\begin{lem}\label{lem:exit_gen2}
Granted Assumption \ref{hyp:exit} and \eqref{eq:cond-general},  there exists $C'>0$ that depends only on $C,\delta$ (and not on $M$) such that 
\[ \sup_{t\in[0,M]} | f(t) - e^{-t}| \ \leqslant \ C'  (1+M) e^{-\delta M/3}\,. \]
\end{lem}
\begin{proof}
The left-hand side is bounded by 2 so that, for fixed values of $C,\delta$, it is sufficient to prove the result for $M$ large enough.  For all $k\in\llbracket 1,\lfloor \beta/s_2\rfloor\rrbracket$, iterating the inequalities \eqref{eq:ineq-general} with $s=s_0/k$, 
\begin{align*}
 e^{-s_0} \ = \ f(s_0) \ & \leqslant \ f\po s_0 - s_0/k \pf \po f\po s_0/k - s_0s_2/\beta\pf + \varepsilon\pf\\
 &  \leqslant \ f\po s_0 - 2s_0/k \pf \po f\po s_0/k - s_0s_2/\beta\pf + \varepsilon\pf^2\\
 & \leqslant \dots  \leqslant \ \po f\po s_0/k - s_0s_2/\beta\pf + \varepsilon\pf^{k}\,,
\end{align*}
 from which
 \begin{equation}\label{eq:demo1gene}
 f\po s_0/k - s_0s_2/\beta\pf  \ \geqslant \ e^{-s_0/k} - \varepsilon\,,
 \end{equation}
 and similarly, provided $f(s_0/k)\geqslant \varepsilon$,
 \begin{align*}
 e^{-s_0} \ = \  f(s_0) \ & \geqslant \  f\po s_0 - s_0/k \pf \po f\po s_0/k \pf-\varepsilon\pf\\
 & \geqslant \ f\po s_0 - 2s_0/k \pf \po f\po s_0/k \pf-\varepsilon\pf^2\\
  & \geqslant \dots \geqslant \  \po f\po s_0/k \pf - \varepsilon\pf^{k}\,,
 \end{align*}
such that
 \begin{equation}\label{eq:demo2gene}
 f\po s_0/k \pf  \ \leqslant  \ \varepsilon + e^{-s_0/k}\,.
 \end{equation}
Besides, the latter inequality is trivial if $f(s_0/k)\leqslant \varepsilon$.

Set $k_0 = \lfloor e^{\delta M/3}/\sqrt C\rfloor$. For $M$ large enough,  $2\leqslant k_0 < \beta/s_2$ and $1/(k_0) \leqslant 1/(k_0-1) -s_2/\beta$. Denoting $u_0 = s_0/k_0$, using first \eqref{eq:demo1gene} with $k=k_0-1$, then the monotonicity of $f$,  and finally \eqref{eq:demo2gene} with $k=k_0$, we get
\[e^{-u_0 \frac{k_0}{k_0-1}} - \varepsilon \ \leqslant \ f\po s_0/(k_0-1) - s_0s_2/\beta\pf \ \leqslant \ f(u_0) \ \leqslant  \varepsilon + e^{-u_0}\,.\]
In the following, we denote by $C'$ various constants that depend only on the parameters $C,\delta$ of \eqref{eq:cond-general}. Since $t\mapsto e^{-t}$ is $1$-Lipschitz continuous on $\R_+$,
\[|e^{-u_0 \frac{k_0}{k_0-1}} - e^{-u_0}| \ \leqslant \ \frac{u_0}{k_0-1} \ = \ \frac{s_0}{k_0(k_0-1)}\,,\] 
and noticing $s_0 \in [\ln(4/3),\ln(4)]$, we get
\[|f(u_0) - e^{-u_0}| \ \leqslant \ C'e^{-2\delta M/3}\]
for some $C'$. By monotonicity again,
\[G(u_0) \ :=\ \sup_{s\in[0,u_0]}|e^{-s} - f(s)| \ \leqslant \ |1-f(u_0)|\vee |e^{-u_0}-1| \ \leqslant \ C'e^{- \delta M/3}\,\]
for some $C'$. Remark that, in particular, we can suppose $M$ large enough so that $f(u_0) \geqslant 1/2 \geqslant \varepsilon$.

 Now, for some $C'$, uniformly in $k\in\llbracket 1,\lceil M/u_0\rceil \rrbracket$ and $r\in[0,u_0]$, using again Lemma \ref{lem:exit_gen1} (more specifically \eqref{eq:ineq-general}),
\[f(k u_0+r) \ \geqslant \ f(r) \po f(u_0) - \varepsilon\pf^k \ \geqslant \ e^{-ku_0-r} - C'(1+M) e^{- \delta M/3}\]
(using that $(1-x)^k \geqslant 1 - 2xk$ for $x$ positive small enough, uniformly in $k$). Similarly, for such $k,r$, 
\[ f\po k(u_0+s_0s_2/\beta) + r\pf  \ \leqslant \ f(r) \po f(u_0) + \varepsilon \pf^{k} \ \leqslant \ e^{-ku_0-r} + C'(1+M) e^{- \delta  M/3}\,,\]
and
\begin{eqnarray*}
f\po k(u_0+s_0s_2/\beta) + r\pf  
& \geqslant & f(ku_0+r)   \po f\po k  s_0 s_2 /\beta \pf - \varepsilon \pf \,.
\end{eqnarray*}
Remark that, for all $k\in\llbracket 1,\lceil M/u_0\rceil \rrbracket$,
\[ k  s_0 s_2 /\beta \ \leqslant\  C'(1+M) e^{-2\delta  M/3} \ \leqslant \ u_0  \]
provided $M$ is large enough (where we used that $s_0\geqslant \ln (4/3)$). In particular,
 \[f\po k  s_0 s_2 /\beta \pf \ \geqslant \  e^{-k  s_0 s_2 /\beta } - G(u_0) \ \geqslant \ 1 - C'e^{-\delta M/3} \]
for some $C'$. Combining the last inequalities, we have obtained, uniformly in $k\in\llbracket 1,\lceil M/u_0\rceil \rrbracket$ and $r\in[0,u_0]$,
\[| f(ku_0+r) - e^{-ku_0-r}| \ \leqslant \ C'(1+M) e^{-\delta M/3} \]
for some $C'$ that depends only on $C,\delta$, which concludes.

\end{proof}

\begin{lem}\label{lem:exit_gen3}
Granted Assumption \ref{hyp:exit} and \eqref{eq:cond-general}, for  $M$ large enough (depending on $\delta,C$), the following holds. For all $y\in\mathcal D$ and all $t\geqslant 0$,
\[\mathbb P_y \po \tau_{\mathcal D^c}  > t \beta/s_0 \pf  \ \leqslant \ 2e^{-t/2}\]
and
\[\mathbb E_y\po \tau_{\mathcal D^c} \pf \ \leqslant \ \frac{4\beta}{s_0}\,.\]
\end{lem}

\begin{proof}

Denote
\[g(t) \ = \ \sup_{y\in\mathcal D} \mathbb P_{y}\po  \tau_{\mathcal D^c}  > t \beta/s_0     \pf\,. \]
The Markov property ensures that $g(t+s)\leqslant g(t) g(s)$ for all $s,t\geqslant 0$. In particular $g$ is non-increasing, and goes exponentially fast toward zero if there exist $t\geqslant 0$ such that $g(t) <1$. Following the same reasoning as for Lemma~\ref{lem:exit_gen1} (in particular considering the same event $A_t$), we see that for all $t\geqslant 0$, all $s\geqslant s_0 s_2 /\beta$ and all $x\in\mathcal D$,
\begin{eqnarray*}
\mathbb P_x\po \tau_{\mathcal D^c}  > (t+s)\beta/s_0\pf & \leqslant & \mathbb P_x(A_t, \tau_{\mathcal D^c}  > (t+s)\beta/s_0 ) + \sup_{y\in\mathcal K}\mathbb P_y\po \tau_{\mathcal D^c}  > s\beta/s_0 - s_2\pf\\
& \leqslant & \varepsilon_2 + \varepsilon_3 + f\po s-s_0s_2/\beta\pf\,.
\end{eqnarray*}
Under Assumption \ref{hyp:exit}, $s_0 s_2 /\beta$ vanishes as $M\rightarrow +\infty.$ So, for $M$ large enough, we can apply the previous inequality with $s=1$ and $t=0$ to get that 
\begin{eqnarray*}
g(1) &  \leqslant &  \varepsilon_2+\varepsilon_3 + f\po 1-s_0s_2/\beta\pf \\
& \leqslant &  C e^{-\delta M} + e^{-1+s_0s_2/\beta} + C'  (1+M) e^{- \delta M/3}
\end{eqnarray*}
with $C'$ given by Lemma~\ref{lem:exit_gen2}. As a consequence, for $M$ large enough, $g(1) \leqslant e^{-1/2}$, and  for all $y\in\mathcal D$ and all $t\geqslant 0$,
\[\mathbb P_y \po \tau_{\mathcal D^c} > t\beta/s_0 \pf  \ \leqslant \ g(\lfloor t\rfloor) \ \leqslant \ e^{-\lfloor t\rfloor/2}\ \leqslant \ 2 e^{-t/2}\,.\]
Finally,
\[\frac{s_0}{\beta}\mathbb E_y\po \tau_{\mathcal D^c} \pf \ = \ \int_0^{+\infty}  \mathbb P_y\po \tau_{\mathcal D^c}  > t\beta/s_0\pf dt \ \leqslant \ 4\,. \]
\end{proof}

\begin{prop}\label{Prop:exit_gene}
Granted Assumption \ref{hyp:exit} and \eqref{eq:cond-general}, there exists $C'>0$ (that depends only on $C,\delta$ but not on $M$) such that for all $M$ large enough the following holds. For all probability measures $\nu$ on $\mathcal D$,
\[\sup_{t\geqslant 0} |\mathbb P_{\nu}\po \tau_{\mathcal D^c} > t\beta/s_0\pf - e^{-t}| \ \leqslant \ \mathbb P_\nu\po \tau_{\mathcal D^c} \leqslant s_2\pf +  C'  (1+M) e^{-\min(\delta/3,1/2)M} \]
and
\[\left|\frac{s_0}{\beta} \mathbb E_{\nu}\po \tau_{\mathcal D^c} \pf - 1 \right| \ \leqslant \ M \mathbb P_\nu\po \tau_{\mathcal D^c} \leqslant  s_2\pf +  C'  (1+M^2) e^{-\min(\delta/3,1/2)M}\,. \]
\end{prop}

\begin{proof}
By conditioning on the initial condition, it is sufficient to prove the result with $\nu=\delta_y$ for any fixed $y\in \mathcal D$. First, for $t\geqslant M$, we simply apply Lemma~\ref{lem:exit_gen3} to get that
\[ |\mathbb P_{y}\po \tau_{\mathcal D^c} > t\beta/s_0\pf - e^{-t}| \ \leqslant \  \ 3 e^{-M/2}\,.\]
Second, for $t\leqslant  s_2 s_0/\beta$, by monotonicity,
\begin{eqnarray*}
|\mathbb P_{y}\po \tau_{\mathcal D^c} > t\beta/s_0\pf - e^{-t}|  & \leqslant & |1 - \mathbb P_{y}\po \tau_{\mathcal D^c} >  s_2\pf | \vee |1 - e^{-s_2 s_0/\beta}|\\
& \leqslant & \mathbb P_{y}\po \tau_{\mathcal D^c} \leqslant   s_2\pf  + C e^{- \delta  M}\,.
\end{eqnarray*}
Third, for $t \in [  s_2 s_0/\beta,M]$, similarly to the proof of Lemma~\ref{lem:exit_gen1} we consider the event $A_0 = \{\forall s\in [0,s_2],\ X_s \notin \mathcal K\}$ and bound
\begin{eqnarray*}
\mathbb P_{y} \po \tau_{\mathcal D^c} >  t\beta/s_0\pf  & \leqslant & \mathbb P_{y} \po \tau_{\mathcal D^c} >  s_2\,,\ A_0\pf  + \mathbb P_{y} \po \tau_{\mathcal D^c} >  t\beta/s_0\ , \ \overline{A_0}\pf \\
& \leqslant & \varepsilon_2 + \sup_{z\in \mathcal K} \mathbb P_{z} \po \tau_{\mathcal D^c} >  t\beta/s_0 -s_2\pf \\
& \leqslant&   \varepsilon_2 + \varepsilon_3 + f(t-s_2s_0/\beta) \ \leqslant \ e^{-t} + C'  (1+M) e^{- \delta M/3}
\end{eqnarray*}
for $M$ large enough so that $\beta\geqslant s_2s_0$ (and thus $|e^{s_2s_0/\beta}-1|\leqslant e s_2s_0/\beta$). Conversely,
\begin{eqnarray*}
\mathbb P_{y} \po \tau_{\mathcal D^c} >  t\beta/s_0\pf  & \geqslant & \mathbb P_{y} \po    \overline{A_0}\pf   \mathbb P_{y} \po \tau_{\mathcal D^c} >  t\beta/s_0\ |\ \overline{A_0}\pf \\
& \geqslant &    \mathbb P_{y} \po \tau_{\mathcal D^c} >  t\beta/s_0\ |\ \overline{A_0}\pf  - \mathbb P_{y} \po    A_0 \pf\\
& \geqslant &   f(t) - \varepsilon_2 - \varepsilon_3 - \mathbb P_{y} \po \tau_{\mathcal D^c} \leqslant s_2 \pf  \\
& \geqslant & e^{-t} -  C'  (1+M) e^{- \delta M/3} - \mathbb P_{y} \po \tau_{\mathcal D^c} \leqslant s_2 \pf \,.
\end{eqnarray*}
This concludes the proof of the first statement. For the second one,
\begin{eqnarray*}
\left|\frac{s_0}{\beta} \mathbb E_{\nu}\po \tau_{\mathcal D^c} \pf - 1 \right| 
 & = & \left| \int_{0}^{+\infty} \mathbb P_{\nu}\po \tau_{\mathcal D^c} >t \beta/s_0 \pf dt - 1 \right| \\
 & \leqslant &  \int_{0}^{M} \left|\mathbb P_{\nu}\po \tau_{\mathcal D^c} >t \beta/s_0 \pf - e^{-t} \right|   dt + e^{-M} + \int_{M}^{+\infty}  \mathbb P_{\nu}\po \tau_{\mathcal D^c} >t \pf dt\,.
\end{eqnarray*}
The first statement and Lemma~\ref{lem:exit_gen3}  conclude.

\end{proof}

We are now ready to prove the main result of this section.

\begin{proof}[Proof of Theorem~\ref{thm:gene-exit}]
The first claim has already been established in Lemma~\ref{lem:exit_gen3}. Denote $a_i=s_0 \mathbb E_{\nu_i}(\tau_{\mathcal D^c})/\beta $ for $i=1,2$ and consider $M$ large enough so that, from Proposition~\ref{Prop:exit_gene}, $a_1\wedge a_2\geqslant 1/2$. Obviously, 
\begin{eqnarray*}
\sup_{t\geqslant 0} |\mathbb P_{\nu_1}\po \tau_{\mathcal D^c} > t \mathbb E_{\nu_1}\po \tau_{\mathcal D^c} \pf \pf - e^{-t}|  & = & \sup_{t\geqslant 0} |\mathbb P_{\nu_1}\po \tau_{\mathcal D^c} > t \beta /s_0 \pf - e^{-t/a_1}|\\
& \leqslant &  \sup_{t\geqslant 0} \po |\mathbb P_{\nu_1}\po \tau_{\mathcal D^c} > t \beta /s_0 \pf - e^{-t}| + |e^{-t}-e^{-t/a_1}|\pf\,.
\end{eqnarray*}
For $t\leqslant M$, the last term is bounded by $M|1-1/a_1|\leqslant 2M[1-a_1|$, and for $t\geqslant M$, it is bounded by $2e^{-M/2}$. Proposition~\ref{Prop:exit_gene} concludes the proof of \eqref{eq:thm-gene-exit-1}. To prove \eqref{eq:thm-gene-exit-2}, we simply bound
\[\left|\frac{a_1}{a_2}-1\right| \ \leqslant 2  |a_1-1|+ 2 |a_2-1|\]
and conclude again with Proposition~\ref{Prop:exit_gene}.
\end{proof}

Finally, as announced, we finish this section by a general argument to establish \eqref{eqdef:generalexit1}.

\begin{prop}\label{prop:eps1}
Let $s_1\geqslant 0$ and $\varepsilon_1,\varepsilon_4\in[0,1]$ be such that \eqref{eqdef:generalexit3} holds and that for all $x,y\in\mathcal K$ there exists a coupling $(X_t,Y_t)_{t\geqslant 0}$ of two processes with respective initial condition $x$ and $y$ such that
\begin{eqnarray}
\mathbb P\po X_{t} = Y_{t}\ \forall t\geqslant s_1\pf & \geqslant & 1 -\varepsilon_4\,.
\end{eqnarray}
Then the condition~\eqref{eqdef:generalexit1} holds with  $\varepsilon_3 = 2 \varepsilon_1+\varepsilon_4$.
\end{prop}

\begin{proof}
Denote respectively $\tau_{\mathcal D^c}$ and $\tilde \tau_{\mathcal D^c}$ the exit times of $X$ and $Y$. Remark that
\[\{\tau_{\mathcal D^c} \geqslant s_1\}\cap\{\tilde \tau_{\mathcal D^c} \geqslant  s_1\}\cap\{X_{t} = Y_{t}\ \forall t\geqslant s_1\} \ \subset\ \{\tau_{\mathcal D^c}=\tilde\tau_{\mathcal D^c}\}\,. \]
As a consequence, for all $t\geqslant 0$,
\begin{eqnarray*}
|\mathbb P\po \tau_{\mathcal D^c} > t\pf -\mathbb P\po \tilde \tau_{\mathcal D^c} > t\pf| & \leqslant & \mathbb E \po |\indiq_{ \tau_{\mathcal D^c} > t}-\indiq_{\tilde  \tau_{\mathcal D^c} > t}|\pf \\
& \leqslant & \mathbb P \po  \tau_{\mathcal D^c} \neq \tilde  \tau_{\mathcal D^c} \pf \\
& \leqslant &  \mathbb P\po \tau_{\mathcal D^c} < s_1\pf + \mathbb P\po \tilde \tau_{\mathcal D^c} < s_1\pf + \mathbb P \po \exists t>s_1\,,\ X_t\neq Y_t\pf\\
& \leqslant & \ 2\varepsilon_1 + \varepsilon_4\,. 
\end{eqnarray*}
\end{proof}

\section{Exponentiality of exit times for the systems of interacting neurons}\label{sec:metastableneurons}

We come back to the study of the process $U^N$ of interacting neurons  introduced in Section \ref{Section-modeldef} and give the proof of Theorem~\ref{theo:exitTimes}. 
 In this section, Assumptions  \ref{ass:lambda} and condition~\eqref{condition_a_b_y0} are enforced, $p_* = \int_0^\infty \lambda g$ where $g$ is the unique positive non-linear equilibrium given by  Corollary~\ref{Cor:uniqueEquilibre}, and $x_\infty = \lambda_*(1-a-b)$ is the unique equilibrium of the auxiliary limit  equation \eqref{eq:limitz} (see Section~\ref{sec:LDP2}). Recall that for $u\in\R_+^N$ and $\delta>0$ we denote $\jump(u)=\sum_{i=1}^N \lambda(u_i)/N$ and $\{\jump \geqslant \delta\}=\{u\in\R_+^N,\ \jump(u) \geqslant \delta\}$ (and similarly for $\{\jump < \delta\}$, etc.).

We wish to apply Theorem~\ref{thm:gene-exit}  in this context. In view of condition \eqref{eqdef:generalexit11} and Proposition \ref{prop:eps1}, it means we have to bound hitting/exit times  for some metastable states, and to be able to couple two processes starting in two different positions in these metastable states. We establish these intermediary results in the next two sections.

%
%



\subsection{Hitting times of metastable sets}

For $\delta>0$ small enough so that $\delta<x_\infty-\delta$, denote $\mathcal K^1_\delta = \{\jump \geqslant x_\infty-\delta\}$ and $\mathcal K^2_\delta = \{p_*-\delta \leqslant \jump \leqslant p_* + \delta\}$.

\begin{prop}\label{prop:exit1}
Grant Assumptions \ref{ass:lambda} and condition~\eqref{condition_a_b_y0}.
\begin{enumerate}
\item  For all $0<\gamma<\delta<x_\infty/2$ there exist $C,T,\theta>0$ such that for all $N\in\mathbb N_*$
\begin{eqnarray*}
\sup_{u\in\R_+^N} \mathbb P_u \po \tau_{\{\jump \leqslant \delta\}} > T\pf & < & 1\\
\sup_{u\in\{\jump \geqslant \gamma\}} \mathbb P_u \po \tau_{\mathcal K_\gamma^1} > T \pf & \leqslant & C e^{-\theta N}\\
\sup_{u\in\mathcal K_\gamma^1}\mathbb P\po \tau_{(\mathcal K_\delta^1)^c} \leqslant e^{\theta N} \pf & \leqslant & C e^{-\theta N}\\
\sup_{u\in\{\jump \geqslant \gamma\}} \mathbb P_u \po \tau_{\mathcal K_\gamma^2} > T \pf & \leqslant & \frac{C}{\sqrt N}\,.
\end{eqnarray*}
\item For all $0<\delta<x_\infty/2$ there exists $\gamma \in (0,\delta)$ and $C>0$ such that for all $N\in\mathbb N_*$
\[\sup_{u\in\mathcal K_\gamma^2}\mathbb P\po \tau_{(\mathcal K_\delta^2)^c} \leqslant N^{1/4} \pf \ \leqslant \ \frac{C}{N^{1/4}}\,.\]
\end{enumerate}

\end{prop}

\begin{proof}
For the first point, remark that, for all $u\in\R_+^N$, 
with a probability larger than $e^{-N\lambda_*}(1-e^{-\lambda_*})^N$, in the time interval $[0,1]$,  all the neurons $i$ with $u_i \geqslant e^{\alpha}\lambda_*/k$ undergo exactly a spike and the other neurons do not spike. In that case, the total number of spikes during  $[0,1]$ is smaller than $N$ so that $U_j^N(1) \leqslant 
e^{\alpha}\lambda_*/k + h$   for all $j\in\cco 1,N\ccf$. Let $t$ be such that
\[\lambda\po e^{-\alpha t} \po e^{\alpha}\lambda_*/k + h\pf \pf \ \leqslant \ \delta \,.\]
Then with positive probability there is no spike during the time $[1,1+t]$ and the process deterministically reaches $\{\jump\leqslant \delta\}$.

For the two next points, consider the auxiliary process $Z^N$ with generator~\eqref{eq:zN} with $Z^N(0) = x_\infty \wedge \jump(u)$. From Proposition~\ref{prop:Z}, $Z^N \leqslant \jump(U^N)$ for all times, in particular $Z^N$ reaches $[x_\infty-\gamma,+\infty)$ after $\jump(U^N)$ and $(0,x_\infty-\delta]$ before $\jump(U^N)$.  The limit equation \eqref{eq:limitz} of $Z^N$ reaches $[x_\infty-\gamma,+\infty)$ from $\gamma$ in a finite time, and the Large Deviation cost $V(x_\infty-\gamma,x_\infty-\delta)$   is positive, so that the Large Deviation result of Theorem~\ref{theo:ldp} concludes the proof of the two first points.

For the fourth point, consider  the settings of Proposition~\ref{prop:propchaos} with $\nu_i = \delta_{u_i}$ for all $i\in\cco 1,N\ccf$. In particular,
\[z_0  \ = \ \frac1N\sum_{i=1}^N \lambda(u_i)  \ \geqslant \ \gamma\,.\]
From Corollary~\ref{Cor:CouplNonLin} applied with $Z(0) \sim 1/N\sum_{i=1}^N\delta_{u_i}$ and $\tilde Z(0) \sim g$,
\[|z_t - p_*|  \ \leqslant\   C_\gamma \kappa^t \lambda_*\,, \]
where we used that $p_* \geqslant \gamma$ since $p_* \geqslant x_\infty$ from Proposition~\ref{PropOinstable}.

In particular there exists $T>0$ (uniform over $u\in\{\jump\geqslant \gamma\}$) such that $|z_T - p_* | \leqslant \gamma/2$. From Proposition~\ref{prop:propchaos}, for all $u\in\{\jump\geqslant \gamma\}$, 
\[\mathbb P_u \po \tau_{\mathcal K_\gamma^2} > T \pf \ \leqslant \ \mathbb P_u \po |z_T - \jump \po U^N(T)\pf| \geqslant  \frac{\gamma}{2} \pf \ \leqslant \ \frac{C}{\sqrt N}\]
for some $C>0$.

For the last point of the proposition, as can be seen by applying Corollary~\ref{Cor:CouplNonLin} as above, for all $\delta>0$ there exists $\gamma\in(0,\delta)$ such that if $z_t = \int_0^{+\infty} \lambda \mu_t$ with $\mu_t$ the law of a process \eqref{eq:limitU} then
\[|z_0 - p_*|\leqslant \gamma \qquad \Rightarrow \qquad |z_t -p_*| < \delta/2 \ \forall t\geqslant 0\,.\]
Moreover, as previously, there exists $T$ such that
\[|z_0 - p_*|\leqslant \gamma \qquad \Rightarrow \qquad |z_T -p_*| \leqslant \gamma/2 \,.\]
If $U^N$ is the process \eqref{eq:EDS_U^N} with initial condition $u\in\{p_*-\gamma\leqslant \jump \leqslant p_*+\gamma\}$, consider the events
\[A_k \ = \ \left\{ |\jump(U^N(s))- p_* | < \delta \ \forall s\in[kT,(k+1)T]\text{ and } |\jump(U^N((k+1)T))- p_* | \leqslant \gamma\right\}\]
for $k\in\mathbb N$. From Proposition~\ref{prop:propchaos},  for all $u\in\{p_*-\gamma\leqslant \jump \leqslant p_*+\gamma\},$
\begin{eqnarray*}
\mathbb P_u(\overline{A_0}) & \leqslant & \mathbb P\po \sup_{s\in[0,T]} |z_s - \jump\po U^N(s)\pf| \geqslant \delta/2\pf + \mathbb P\po  |z_T - \jump\po U^N(T)\pf| \geqslant \gamma/2\pf  \ \leqslant \ \frac{C}{\sqrt N}
\end{eqnarray*}
for some $C$ independent from $u$. By the Markov property, for all $K\in\mathbb N_*$ and $u\in\{p_*-\gamma\leqslant \jump \leqslant p_*+\gamma\},$
\[\mathbb P_u\po \tau_{(\mathcal K_\delta^2)^c} \leqslant KT\pf  \ \leqslant \ \mathbb P_u \po \overline{A_0} \pf  +\mathbb P_u \po A_0 \cap \overline{A_1}\pf + \ldots +  \mathbb P_u \po A_0 \cap \ldots \cap A_{K-2} \cap \overline{A_{  K-1}}\pf \ \leqslant \ \frac{KC}{\sqrt N}\,.\]
Conclusion follows from the choice $K=\lceil N^{1/4}/T\rceil$.
\end{proof}

\subsection{Coupling two systems of interacting neurons}

The coupling argument for two systems of interacting particles with different initial conditions partially mimics those of the non-linear processes developed in Section~\ref{sec:couplage-non-lineaire}. In particular, before coupling the processes, we start by coupling their jump rates.

\begin{prop}\label{prop:couple-Lambda-N}
Grant Assumptions  \ref{ass:lambda} and condition~\eqref{condition_a_b_2}. There exist $C,\theta,\gamma>0$ and $\kappa\in(0,1)$ such that the following holds. For all $u_0,\tilde u_0\in\{\jump\geqslant x_\infty-\gamma\}$, the synchronous coupling of $(U^N,\tilde U^N)$ with initial condition $(u_0,\tilde u_0)$ satisfies  for all $N\in\mathbb N_*$ and all $t\in[0,N]$
\begin{eqnarray*}
\mathbb E\po \sum_{i=1}^N |\lambda(U_i^N(t)) - \lambda(\tilde U_i^N(t))|\pf & \leqslant & C \po \kappa^t N + e^{-\theta N}\pf\,. 
\end{eqnarray*}
\end{prop}

\begin{proof}
Let $\delta >0$ be small enough so that, considering $t_\delta$ given by \eqref{eq:to}, then
\[\kappa_\delta \ := \ \lambda_* t_\delta  e^{\lambda_*t_\delta }  \po 1 + kh  t_\delta  e^{2(2kh+\lambda_*)t_\delta}\pf  <1\,.\]
It is indeed possible to do so since, as $\delta$ vanishes, $\kappa_\delta$ goes to
\begin{eqnarray*}
\kappa_0 & = & \frac{b}{a}\ln\po 1+\frac{a}{1-2a-b}\pf e^{\frac{b}{a}\ln\po 1+\frac{a}{1-2a-b}\pf} \po 1+ \frac1a \ln\po 1+\frac{a}{1-2a-b}\pf e^{\frac{4+2b}{a} \ln\po 1+\frac{a}{1-2a-b}\pf}\pf\\
& < & \frac{b}{1-2a-b} \exp \po \frac{b}{1-2a-b} \pf \po 1 + \frac{1}{1-2a-b} \exp\po \frac{4+2b}{1-2a-b}\pf \pf \ \leqslant \ 1\,,
\end{eqnarray*}
where we used that $\ln(1+x)<x$ for all $x>0$.

Let $\gamma \in (0,\delta)$ and take $u_0,\tilde u_0\in\{\jump\geqslant x_\infty-\gamma\}$. Considering these two different thresholds ($\delta$ and $\gamma$) is motivated by the following reason: starting with an average jump rate above the level $x_\infty - \gamma$, using the comparison with the auxiliary process $Z^N$, we will get that the average jump rate stays with high probability above $x_\infty -\delta$ during the time interval $[0,N]$. This replaces the argument in the limit non-linear case of Section~\ref{sec:couplage-non-lineaire} where the process deterministically stays in $\{\jump \geqslant x_\infty -\delta\}$ if it started there.

For all $i\in\cco 1, N\ccf$ and $t\geqslant 0$, denote $W_i(t) = |\lambda(U_i^N(t))-\lambda(\tilde U_i^N(t))|$.

\bigskip
 
 \textbf{Step 1.} Fix $i\in \cco 1,N\ccf$ and $t\geqslant 0$.

{\it Case 1.} First, suppose that $t\leqslant t_\delta$ and that there is no spike in the time interval $[0,t]$ for both $U^N_i$ and $\tilde U_i^N$.  In the absence of spike for the $i^{th}$ neuron, $W_i$ only increases (at most by $kh/N$) when there is an asynchronous spike for a pair $(U_j^N,\tilde U_j^N)$ for $j\neq i$, which happens at rate $ W_j$.  As a consequence, for $t\leqslant t_\delta$,
\[\E\po W_i(t)\indiq_{\text{no spike of $i$  in }[0,t]}\pf   
\ \leqslant \ \E\po W_i(0)\pf +  \frac{kh}N \int_{(t-t_\delta)_+}^t \sum_{j\neq i} \E\po W_j(s)\pf ds\,.\]

{\it Case 2.} Second, suppose that there is an asynchronous spike for $(U_i^N,\tilde U_i^N)$ in $[(t-t_\delta)_+,t]$. In that case we simply bound $W_i(t)\leqslant \lambda_*$ and then
\begin{eqnarray*}
\mathbb E\po W_i(t)\indiq_{\text{asynchronous for $i$ in }[(t-t_\delta)_+,t]} \pf & \leqslant & \lambda_* \mathbb P\po \text{asynchronous for $i$ in }[(t-t_\delta)_+,t]\pf\\
& \leqslant & \lambda_*\int_{(t-t_\delta)_+}^t \mathbb E \po W_i(s) \pf ds \,.
\end{eqnarray*}

{\it Case 3.} Third, denoting $D_t^i$ the first spike of either $U_i^N$ or $\tilde U_i^N$ after time $(t-t_\delta)_+$, consider the event $E_t^i = \{D_t^i\leqslant t \,, U^N(D_t^i)=\tilde U_i^N(D_t^i)=0\}$. In other words, under $E_t^i$, $D_t^i$ corresponds to a synchronous spike, in particular $W_i(D_t^i)=0$.    As in Case 1, after the time $D_t^i$ and in the absence of asynchronous jumps for $i$, $W_i$ only increases (at most by $kh/N$) when there is an asynchronous spike for a pair $(U_j^N,\tilde U_j^N)$ for $j\neq i$. 
More precisely, writing $F_t^i=\{$no asynchronous spike for $i$ in $[(t-t_\delta)_+,t]\}$ then,  almost surely,
\begin{eqnarray}\label{eq:WiEi}
W_i(t) \indiq_{E_t^i}\indiq_{F_t^i}  & \leqslant & \indiq_{E_t^i}\frac{kh}{N}\int_{D_t^i}^t \sum_{j\neq i}\int_{\R_+}|\indiq_{\{z\leqslant \lambda(U^N_j(s))\}}-\indiq_{\{z\leqslant \lambda(U^N_j(s))\}}|\pi^j(dz,ds)\,.
\end{eqnarray}

Remark that, by comparison with the non-linear case of Theorem~\ref{theo:CouplNonLin}, there is an additional difficulty here, which is that $E_t^i$ is not independent from the asynchronous jumps of the neurons $j\neq i $ after time $D_t^i$. On the other hand we cannot simply bound the indicator of $E_t^i$ by $1$ because then we would miss a factor $t_\delta$ that is crucial to obtain at the end a contraction with a rate $\kappa<1$.  We will come back to this question in Step 2 below but, for now, indeed we simply bound the indicator by $1$ to get
\[\E\po W_i(t)\indiq_{E_t^i}\indiq_{F_t^i}\pf   
\ \leqslant \  \frac{kh}N \int_{(t-t_\delta)_+}^t \sum_{j\neq i} \E\po W_j(s)\pf ds\,.\]
{\it Conclusion of Step 1.} At this point we have established that for all $t\leqslant t_\delta$,
\[\sum_{i=1}^N \E\po W_i(t)\pf \ \leqslant \ \sum_{i=1}^N  \E\po W_i(0)\pf + (2kh+\lambda_*)\int_0^t \sum_{i=1}^N \E\po W_i(s)\pf ds\]
and thus
\begin{equation}\label{eq:Gronwall}
\sum_{i=1}^N \E\po W_i(t)\pf   \leqslant  e^{(2kh+\lambda_*)t}\sum_{i=1}^N \E\po W_i(0)\pf\,. 
\end{equation} 

\bigskip

\textbf{Step 2.} Fix $i\in\cco 1,N\ccf$ and $t\geqslant 0$. We now tackle the issue raised in Case 3 of Step 1 by considering a process $(V,\tilde V)$ similar to $(U^N,\tilde U^N)$ except that the spike of the $i^{th}$ neurons has no effect on the rest of the system. More precisely, we put, for all $ j \neq i, $ 
\[ d V_j (t) = - \alpha   V_j  (t) dt + \frac{h}{N}  \sum_{k \neq j, i  }   \int_{\R_+} \indiq_{\{ z \le \lambda ( V_k ({t-}) ) \}} \pi^k ( dt, dz) 
   -  \int_{\R_+} V_j ({t-} ) \indiq_{\{ z \le \lambda ( V_t ({t-}) ) \}} \pi^j ( dt, dz),\]
and 
\[ dV_i (t) =- \alpha  V_i  (t) dt + \frac{h}{N}  \sum_{k \neq  i  }  \int_{\R_+} \indiq_{\{ z \le \lambda ( V_k ({t-}) ) \}} \pi^k ( dt, dz) 
  - \int_{\R_+}V_i ({t-} ) \indiq_{\{ z \le \lambda ( V_i ({t-}) ) \}} \pi^i ( dt, dz).\]
The process $ \tilde V$ follows the same dynamic with the same Poisson noise. We initialise these auxiliary processes at time $(t- t_\delta)_+$ and let them start from $ V ((t- t_\delta)_+) = U^N ((t- t_\delta)_+ ) $ and $ \tilde V ((t- t_\delta)_+ ) = \tilde U^N ((t- t_\delta)_+ ) $. For all $j\in\cco 1,N\ccf$ and $s\geqslant (t-t_\delta)_+$, denote $\tilde W_j(s) = |\lambda(V_j(s))-\lambda(\tilde V_j(s))|$. The arguments of Step 1 are straightforwardly adapted to the process $(V,\tilde V)$ to get that, for all $s\geqslant (t-t_\delta)_+$,
\begin{eqnarray}\label{eq:GronwallVtildeV}
\sum_{j=1}^N \E\po \tilde W_j(s)\pf  & \leqslant & e^{(2kh+\lambda_*)s}\sum_{j=1}^N \E\po \tilde  W_j \po (t-t_\delta)_+\pf\pf\,. 
\end{eqnarray} 
Now consider again the context of Case 3 in Step 1, namely the event $E_t^i$. Taking in \eqref{eq:WiEi} the conditional expectation with respect to $\mathcal F_{D_t^i} = \sigma\{(U^N(s),\tilde U^N(s))_{s\leqslant D_t^i}\}$, using the strong Markov property and then \eqref{eq:Gronwall}, we get
\begin{eqnarray*}
\E \po  \indiq_{  E_t^i} \indiq_{F_t^i}W_i(t)\pf & \leqslant &  
\frac{ k h}{N}  \E \left(   \indiq_{ E_t^i }  \E_{ (U^N(D_t^{i}), \tilde U^N(D_t^{i}) )}\left[ \int_0^{ t - T }  \sum_{j=1}^N  W_j(s) ds\right]_{T = D_t^i  }  \right)\\ 
& \leqslant &  
\frac{ k h}{N}  \E \left(   \indiq_{ E_t^i } \sum_{j=1}^N  W_j(D_t^i) \int_0^{ t - D_t^i }e^{(2kh+\lambda_*)s}ds   \right) \,.
\end{eqnarray*}
Before time $D_t^i$, by design, $(U^N,\tilde U^N)=(V,\tilde V)$, in particular $W_j=\tilde W_j$ for all $j\in\cco 1,N\ccf$. At time $D_t^i$, $W_i(D_t^i)=0\leqslant \tilde W(D_t^i)$ and, since $\lambda$ is a concave function, for all $j\neq i$,
\begin{eqnarray*}
W_j(D_t^i) & = & \left|\lambda\po U_j^N(D_t^i-)+\frac{kh}N\pf-\lambda\po \tilde U_j^N(D_t^i-)+\frac{kh}N\pf\right|\\
 & \leqslant &  \left|\lambda\po U_j^N(D_t^i-)\pf-\lambda\po \tilde U_j^N(D_t^i-)\pf\right| \  = \ \tilde W_j(D_t^i- )= \tilde W_j(D_t^i )\,.
\end{eqnarray*}
As a consequence, 
\begin{eqnarray*}
\E \po  \indiq_{  E_t^i} \indiq_{F_t^i} W_i(t)\pf 
& \leqslant &  
\frac{ k h}{N}  \E \left(   \indiq_{ D_t^i \leqslant t } \sum_{j\neq i} \tilde W_j(D_t^i) \int_0^{ t - D_t^i }e^{(2kh+\lambda_*)s}ds   \right) \\
& \leqslant &  
\frac{ k h}{N} \int_{(t-t_\delta)_+}^t \E\po  \sum_{j\neq i} \tilde W_j(u) \pf \lambda_* \int_0^{ t - u }e^{(2kh+\lambda_*)s}ds  du
\,,
\end{eqnarray*}
where we have used that the density of the conditional law of $D_t^i$ with respect to $(V(s),\tilde V(s))_{s\geqslant(t-t_\delta)_+}$ is always bounded by $\lambda_*$. Finally, using \eqref{eq:GronwallVtildeV}  
\begin{eqnarray*}
\E \po  \indiq_{  E_t^i} \indiq_{F_t^i} W_i(t)\pf 
& \leqslant &  
\frac{ k h}{N} \int_{(t-t_\delta)_+}^t \E\po  \sum_{j\neq i}  W_j \po (t-t_\delta)_+\pf\pf \lambda_* e^{(2kh+\lambda_*)u}\int_0^{ t - u }e^{(2kh+\lambda_*)s}ds  du\\
& \leqslant & \frac{kh}{N} \lambda_*t_\delta^2 e^{2(2kh+\lambda_*)t_\delta} \E\po  \sum_{j\neq i}  W_j \po (t-t_\delta)_+\pf\pf \,.
\end{eqnarray*}
As will be clear in Step 4 below, here we have solved the issue raised in Case 3 of Step 1. 

\bigskip

\textbf{Step 3.} In Step 1, we have considered the case where there is no spike in $[(t-t_\delta)_+,t]$ only for $t\leqslant t_\delta$.
Now let $t>t_\delta$, and $i\in\cco 1,N\ccf$. Considering the event 
\[A_t^i = \left\{\text{there is no spike for the $i^{th}$ neurons in $[t-t_\delta,t]$, and }\lambda(U_i^N(t))\neq \lambda(\tilde U_i^N(t))\right\}\,,\]
we simply bound
\[\E\po W_i(t)\indiq_{\text{no spike in }[t-t_\delta,t]}\pf \ \leqslant \ \lambda_* \mathbb P\po A_t^i\pf\,.\]
We now bound $\mathbb P\po A_t^i\pf$ for $t\in[t_\delta, N]$. Consider the event 
\[D_t=\left\{\forall s \in [(t-t_\delta)_+, t ]\,,\ \jump\po U^N(s)\pf \wedge \jump\po \tilde U^N(s)\pf \geqslant x_\infty - \delta\right\}\,.\]
Then
\[\mathbb P(A_t^i) \ \leqslant \ \mathbb P(A_t^i \cap D_t) + \mathbb P\po \bar D_t\pf\,.\]
Consider the auxiliary process $Z^N$ as defined in Section~\ref{section:auxiliaryprocess} with initial condition $Z^N(0) =  x_\infty-\gamma$ and synchronously coupled with $U^N$. According to Proposition~\ref{prop:Z}, $Z^N$ stays below $\jump(U^N)$ for all times. In particular, under the event $\bar D_t $, it reaches the level $x_\infty-\delta$ in a time smaller than $t\leqslant N$. As established in the proof of Proposition~\ref{prop:ldp1}, the Large Deviation cost $V(x_\infty-\gamma,x_\infty-\delta)$ is positive, so that 
\[\mathbb P \po \bar D_t\pf \ \leqslant \ K e^{-\theta N}\]
for some $K,\theta>0$ (where $\tilde U^N$ has been treated similarly to $U^N$).

It remains to bound $\mathbb P(A_t^i \cap D_t)$. Using that $\{U_i^N(t))\wedge \lambda(U_i^N(t)) \geqslant \lambda_*/k\} \subset \{\lambda(U_i^N(t))=\lambda(U_i^N(t))\}$, we bound
\begin{multline*}
\mathbb P(A_t^i \cap D_t) \ \leqslant \ \mathbb P \po \jump\po U^N(t-t_\delta)\pf \geqslant x_\infty-\delta\,,\ \text{no spike for $U^N_i$ during }[t-t_\delta,t]\,,\  U_i^N(t) < \lambda_*/k\pf \\
\ + \ \mathbb P \po \jump\po \tilde U^N(t-t_\delta)\pf \geqslant x_\infty-\delta\,,\ \text{no spike for $\tilde U^N_i$ during }[t-t_\delta,t]\,,\ \tilde U_i^N(t) < \lambda_*/k\pf\,.
\end{multline*}
The two terms being similar, we only treat the first one. Note that if $U_i^N$ presents no spike in $[t-t_\delta,t]$ then on this time interval the $i^{th}$ neuron has no influence on the rest of the system. In other words, $V_j(s)=U_j^N(s)$ for all $j\neq i $ and $s\in[t-t_\delta,t],$ where $V$ has been introduced in Step 2.   Remark that $V^{N-1}:=(V_j)_{j\neq i}$ is exactly a system of interacting neurons with generator \eqref{eq:generator0}, but with only $N-1$ neurons. Moreover,
\[\frac1N \sum_{j\neq i }  \lambda \po V_j(t-t_\delta)\pf \ = \ \frac1N \sum_{j\neq i } \lambda\po U_j^N(t-t_\delta)\pf \ \geqslant \ x_\infty - \delta - \frac{\lambda_*}{N}\,.\]
 Consider $Z^{N-1}$ the auxiliary process synchronously coupled with $V^{N-1}$ and initialized at time $t-t_\delta$ by $Z^{N-1}(t-t_\delta) = x_\infty - \delta - \lambda_*/N$. Recall that, as we saw in the proof of Proposition~\ref{prop:Z},  $Z^{N-1}$ jumps only when $V^{N-1}$ jumps. 
 
Each time $V^{N-1}$ jumps, $U_i^N$ is increased by $h/N$. As a consequence,  in the absence of spike for the $i^{th}$ neuron, $U_i^N(t) \geqslant Y^{N-1}$ where $Y^{N-1}$ is the process that solves   $\dot Y^{N-1} \ = \ -\alpha Y^{N-1}$ between jumps of $Z^{N-1}$, is increased by $h/N$ at each jump of $Z^{N-1}$, in other words
\[d Y^{N-1}(t) \ = \ -\alpha Y^{N-1}(t) dt +\frac hN \int_{\R_+} \indiq_{\{z\leqslant Z^{N-1}(t-)\}}\sum_{j\neq i}\pi^j(dz,dt)\,,\] 
 and is initialized at $Y^{N-1}(t-t_\delta) = 0$. As $N\rightarrow +\infty$, $(Z^{N-1},Y^{N-1})$ converges towards the solution of  
\[\left\{\begin{array}{rcl}
\dot z_s  & =& -\alpha z_s + G(z_s)f(z_s)\\
\dot y_s & = & -\alpha y_s + h z_s
\end{array}\right.\]
with initial position $(z_{t-t_\delta},y_{t-t_\delta})=(x_\infty-\delta,0)$. Recall that $t_\delta$ is by definition the time for the solution of 
\[ \dot x_s \ = \ -\alpha x_s + h (x_\infty-\delta)\,,\qquad x_{t-t_\delta} = 0\]
to reach the threshold $\lambda_*/k$ (see the  proof of Theorem~\ref{theo:CouplNonLin}). Since $\dot y_s > \dot x_s$ for all $s>t-t_\delta$, $y_s$ reaches this threshold in a time $s_\delta < t_\delta$. Finally, the arguments of Section~\ref{sec:ldp1} to obtain a Large Deviation Principle for $Z^N$ are straightforwardly adapted to the process $(Z^{N-1},Y^{N-1})$ to get that
\[\mathbb P \po Y^{N-1}(t) < \lambda_*/k\pf \ \leqslant \ K e^{-\theta N}\]
for some $K,\theta>0$.

As a conclusion of Step 3, we have proven that, for some $K,\theta>0$, for all $t\in[t_\delta, N]$,
\[ \sum_{i=1}^N \E\po W_i(t)\indiq_{\text{no spike for the $i^{th}$ neurons in }[t-t_\delta,t]}\pf \ \leqslant \ K e^{-\theta N}\,.\]

\bigskip

\textbf{Step 4.} This is now similar to the second step of the proof of Theorem~\ref{theo:CouplNonLin}. Gathering the results of Step 1, Step 2 and Step 3, and denoting 
\[f(t) \ = \ \sum_{i=1}^N \E\po W_i(t)\pf\,,\]
we have obtained that for all $t\in[0,N]$,
\[f(t) \ \leqslant  \  \lambda_*  \int_{(t-t_\delta)_+}^t f(s) ds + \co \tilde \nu_\delta f\po t-t_\delta\pf + K e^{-\theta N}\cf\indiq_{t> t_\delta} +  \co f(0) + 2kh \int_0^t f(s) ds\cf \indiq_{t\leqslant t_\delta} \]
with $\tilde \nu_\delta = kh \lambda_*t_\delta^2 e^{2(2kh+\lambda_*)t_\delta}$. Similarly to Step 2 of the proof of Theorem~\ref{theo:CouplNonLin}, we deduce that for all $n\in\mathbb N$ all $t\in [nt_{\delta},(n+1)t_\delta]$ with $t\leqslant N$
$$ f( t) \ \leqslant \ e^{(2kh+\lambda_*)t_\delta }\kappa_\delta^n f (0) + K' e^{ - \delta N}\sum_{k=0}^{n-1}\kappa_\delta^k $$
for some $K'>0$ and $\kappa_\delta<1$ by choice of $\delta$ at the beginning of the proof. As a conclusion, for all $t\leqslant N$,
\[f(t)  \ \leqslant \ \kappa_\delta^{t-1}   e^{(2kh+\lambda_*)t_\delta }  f (0) + \frac{K'}{1-\kappa_\delta} e^{ - \delta N}\]
and $f(0)\leqslant N\lambda_*$.

\end{proof}
 
 \begin{prop}\label{prop:exit2}
Grant Assumptions \ref{ass:lambda} and condition~\eqref{condition_a_b_2}. For all $\zeta\in(0,1]$ there exist $C,\theta, \delta >0$  such that, for all $N\in\mathbb N_*$, for all $u_0,\tilde u_0\in\{\jump\geqslant x_\infty-\delta\}$, the synchronous coupling $(U^N,\tilde U^N)$ of two Markov processes with generator \eqref{eq:generator0} and respective initial conditions $u_0$ and $\tilde u_0$ satisfies
\begin{eqnarray*}
\mathbb P\po U^N(N^\zeta) \neq \tilde U^N(N^\zeta) \pf & \leqslant &  C e^{-\theta N^\zeta}\,.
\end{eqnarray*}
\end{prop}

\begin{proof} 
Take $\delta$ as in the proof of Proposition \ref{prop:couple-Lambda-N}. Without loss of generality, we assume that $ N^\zeta > 2 t_\delta.$ Let $(\hat  Z,\hat Y)$ be the Markov process that solves
\begin{eqnarray*}
d \hat Z(t) & =  & -\alpha \hat Z(t)dt + \po \po z_N\wedge m_N(\hat Z(t-))\pf - \hat Z(t-)\pf \int_{\R_+} \indiq_{\{z \leqslant \hat Z(t-)\}} \sum_{i=1}^N\pi^i(dt,dz)\\
d \hat Y(t) & = & -\alpha \hat Y(t) dt + \frac{h}{N}  \int_{\R_+} \indiq_{\{z \leqslant \hat Z(t-)\}} \sum_{i=1}^N\pi^i(dt,dz)\,,
\end{eqnarray*}
initialized  at time $N^\zeta/2 $ at $(x_\infty-\delta, 0)$. For all $i\in\cco 1,N\ccf$, introduce
$$ K^i_t := \int_{ [0, t]} \int_{\R_+} 1_{\{ z \le \lambda^*  \}} \pi^i ( ds, dz)  \,,$$ 
and consider the events
\begin{eqnarray*}
A_N & = & \{\text{no asynchronous spike of any of the neurons in the time interval }[N^\zeta/2 ,N^\zeta]\},\\
B_N &= & \{\hat Y (s) \geqslant \lambda_*/k \mbox{ for all } s \in [N^\zeta/2+ t_\delta, N^\zeta ]  \}\,,\\
C_N &=& \bigcap_{j=1}^N \{ K^j \mbox{ has at least one jump in  } [N^\zeta/2 + t_\delta, N^\zeta  ] \} \\
D_{N} &= & \left\{\jump\po U^N(N^\zeta/2)\pf \wedge \jump\po \tilde U^N(N^\zeta/2)\pf \geqslant x_\infty - \delta\right\}\,. 
\end{eqnarray*}
Then 
$$A_N \cap B_N \cap C_N \cap D_N \subset\{U^N(N^\zeta)=\tilde U^N(N^\zeta)\}.$$ 
Indeed, if there is no asynchronous spike during $[N^\zeta/2,N^\zeta]$ in the whole system then, as soon as a pair $(U_i^N,\tilde U_i^N)$ undergoes a synchronous spike in this time interval, they evolve synchronously afterwards and thus $U_i^N(N^\zeta)=\tilde U_i(N^\zeta)$. Moreover, under $D_N$, $\hat Z \leqslant \jump(U^N)\wedge \jump(\tilde U^N)$ after time $ N^\zeta/2, $  and thus $\hat Y \leqslant U_i^ N \wedge \tilde U_i^N$ for all $i\in\cco 1,N\ccf$ up to the first  spike of $(U_i^N,\tilde U_i^N)$ occurring after time $N^\zeta/2$. In particular, under $B_N$, it means that $\lambda(U_i^N)=\lambda(\tilde U_i^N) = \lambda_*$ for all $t\geqslant N^\zeta/2+t_\delta$, so that a jump of $K^i$ in this time interval is  a synchronous spike of $(U_i^N,\tilde U_i^N)$, which concludes.

From Proposition \ref{prop:couple-Lambda-N},
\[\mathbb P\po \overline{A_N} \pf \ \leqslant \int_{\frac{N^\zeta}{2}}^{N^\zeta} \sum_{i=1}^N \mathbb E \po \left| \lambda\po U_i^N(s) \pf -\lambda\po  \tilde U_i^N(s) \pf\right|  \pf ds \ \leqslant \ C N^\zeta\po \kappa^{N^\zeta/2}N + e^{ - \theta N^\zeta/ 2}\pf \,. \]
Using the Large Deviations Principle of Theorem \ref{theo:ldp} as in Step 3 of the proof of Proposition \ref{prop:couple-Lambda-N},
\[\mathbb P \po \overline{B_N}\pf + \mathbb P \po \overline{D_N}\pf \ \leqslant \ C e^{-\theta N}\]
for some $C,\theta>0$.
Finally, 
$$ \mathbb P ( \overline{C_N} ) \le N e^{ - \lambda_* (\frac{N^\zeta}{2} - t_\delta) }\,.$$
Summing these three inequalities concludes the proof. 
\end{proof}

\subsection{Conclusion}\label{sec:conclusion-expo}

\begin{proof}[Proof of Theorem~\ref{theo:exitTimes}]
First, since $\mathcal D\subset \{\jump\geqslant \gamma\}$ for some $\gamma>0$ in both cases, from Proposition~\ref{prop:exit1}, there exists $T>0$ such that $\xi = \sup_{u\in\R_+^N} \mathbb P_u \po \tau > T\pf < 1$. As a consequence, by the Markov property, for all $k\in\N$, $N\geqslant 1$ and $u\in\R_+^N$,
\[ \mathbb P_u \po \tau > kT\pf \ \leqslant \ \xi^k\,.\]
In particular $\tau$ is $\mathbb P_u$-almost finite and $\sup_{u\in\R_+^N}\mathbb E_u(\tau) \leqslant T/(1-\xi)$.

Second, apply Theorem \ref{thm:gene-exit} using Propositions  \ref{prop:eps1}, \ref{prop:exit1} and \ref{prop:exit2}.

More precisely, in case (1) we chose any $x_0\in\mathcal K$,  $s_1=N$, $s_2=T$, $\varepsilon_1  = \varepsilon_2= \varepsilon_4 = C e^{-\theta N}$ (and thus, by Proposition \ref{prop:eps1}, $\varepsilon_3 = 3C e^{-\theta N}$). For $N$ large enough, Lemma~\ref{lem:beta-exitgene} applies, and if $\beta$ is such that $\mathbb P_{x_0}(\tau >\beta) \in[1/4,3/4]$ then from Proposition~\ref{prop:exit1} $\beta \geqslant e^{\theta N}$. As a consequence, Assumption~\ref{hyp:exit} holds with $M=N$.

In case (2), we chose  any $x_0\in\mathcal K$,  $s_2=T$, $s_1=N^{1/4}$, $\varepsilon_2  = C/\sqrt N $, $\varepsilon_1=C/N^{1/4}$, $\varepsilon_4 = C e^{-\theta N^{1/4}}$ (and thus, by Proposition~\ref{prop:eps1}, $\varepsilon_3 = 3C/N^{1/4}$).   For $N$ large enough, Lemma~\ref{lem:beta-exitgene} applies, and if $\beta$ is such that $\mathbb P_{x_0}(\tau >\beta) \in[1/4,3/4]$ then from Proposition~\ref{prop:exit1} $\beta \geqslant N^{1/4}$. As a consequence, Assumption~\ref{hyp:exit} holds with $M=\ln N$ and $\delta = 1/4$.
\end{proof}

\section*{Acknowledgements}

This research has been conducted as part of FAPESP project Research, Innovation and Dissemination Center for Neuromathematics (grant 2013/07699-0) and of the  ANR project ChaMaNe ANR-19-CE40-0024. P. Monmarch\'e acknowledges partial funding from the French ANR grants EFI (Entropy, flows, inequalities, ANR-17-CE40-0030) and METANOLIN (Metastability for nonlinear processes, ANR-19-CE40-0009).

\end{document}